\documentclass{article}
\usepackage{graphicx}
\usepackage[utf8]{inputenc}
\usepackage[T1]{fontenc}
\usepackage{subcaption}
\usepackage{amsmath,amssymb,lmodern}
\usepackage{amsmath}
\usepackage{amsthm}
\usepackage{caption}
\usepackage{color,soul}
\usepackage{float} %Forcing exact graphic placement - H
\providecommand{\keywords}[1]{\textbf{Keywords} #1}
\newtheorem{remark}{Remark}
\newtheorem{result}{Result}
\newtheorem{theorem}{Theorem}

\begin{document}

\title{Subgrid multiscale stabilized finite element analysis of  non-Newtonian Power-law model fully coupled with Advection-Diffusion-Reaction equations}

\author{Manisha Chowdhury, B.V. Rathish Kumar\thanks{ Email addresses: chowdhurymanisha8@gmail.com(M.Chowdhury); drbvrk11@gmail.com (B.V.R. Kumar)   } }
      
\date{Indian Institute of Technology Kanpur \\ Kanpur, Uttar Pradesh, India}

\maketitle

\begin{abstract}
This article presents stability and convergence analyses of subgrid multiscale stabilized finite element formulation of non-Newtonian power-law fluid flow model strongly coupled with variable coefficients Advection-Diffusion-Reaction ($VADR$) equation. Considering the highly non-linear viscosity coefficient as solute concentration dependent makes the coupling two way. The stabilized formulation of the transient coupled system is developed based upon time dependent subscales, which ensures inherent consistency of the method. The proposed algebraic expressions of the stabilization parameters appropriately shape up the apriori and aposteriori error estimates. Both the shear thinning and shear thickening properties, indicated by different power-law indices are properly highlighted in theoretical derivations as well as in numerical validations. The numerical experiments carried out for different combinations of small and large Reynolds numbers and power-law indices establish far better performance of time dependent $ASGS$ method in approximating the solution of this coupled system for all the cases over the other well known stabilized finite element methods.
\end{abstract}


\section{Introduction}
Importance of studying transportation of non-Newtonian fluids lies in its wide range of applications in the fields of chemical engineering (\cite{2}, \cite{18}, \cite{19}), biomedical engineering (\cite{17},\cite{22},\cite{29},\cite{32},\cite{37}), fluid mechanics (\cite{30},\cite{31},\cite{33},\cite{36}), petroleum engineering \cite{38}, environmental sciences etc. in modeling various contemporary real life based problems. Non-Newtonian fluid flow obeying power-law coupled with $VADR$ equation plays an important role in effectively designing these complex phenomena mathematically. Several complex properties of non-Newtonian fluids which includes shear thinning, shear thickening etc. expressed in terms of its highly non-linear viscosity coefficient makes it tremendously challenging to handle. Therefore being an attraction of study numerous simulation based research works (\cite{1},\cite{21},\cite{23}-\cite{26}) have been carried out with this mathematical model of non-Newtonian fluids.\vspace{1mm}\\
In spite of such a huge amount of numerical studies, in literature there is a shortage of mathematical analysis of the numerical methods employed to find approximate solutions of non-Newtonian fluid flow model. $Barrett$ and $Liu$ have found out energy type error bounds for non-Newtonian power-law model using the standard $Galerkin$ finite element method in \cite{7} and \cite{8}. The $Galerkin$ method though an excellent interpolation scheme, suffers from various numerical instabilities in the case of convection dominated flows admitting high Reynolds number ($Re$) and for few easy to implement finite element spaces which are incapable in satisfying compatible $inf$-$sup$ condition for velocity-pressure approximate spaces. Contrarily the stabilized finite element methods, studied in \cite{5}, \cite{6},\cite{12}, \cite{13}, \cite{28}, perform extremely well in approximating the solution of Newtonian fluids for any pair of velocity pressure spaces circumventing the $inf$-$sup$ condition and even for convection dominated flows (\cite{11},\cite{20}, \cite{27}). Use of stabilized methods to solve non-Newtonian fluid flow problems is not new, for instance $Aguuirre$ $et.$ $al.$ in \cite{3} has studied lid driven cavity problem for both steady and unsteady power-law fluid flows at high $Re$ using subgrid multiscale stabilized finite element method for quasi-static subscales (where the subscales are not time dependent).  Adequately well performance of the $VMS$ type stabilized FEM motivates us to study its time dependent approach, introduced for Navier-Stokes equations in \cite{15} for coupled non-Newtonian power-law-transport system. This method begins with an additive decomposition of the continuous solution into coarse or resolvable and fine or unresolvable scales, also known as subgrid scales. The resolvable scales are chosen to be numerically computable and the subgrid scales are analytically expressed in terms of the coarse  scales. Finally the elimination of the subgrid scales results into the stabilized formulation. This mathematical nesting of coarse and fine scales makes the method inherently consistent, accurate and stable for lesser refined spatial discretizations. It is noteworthy to mention here that the consideration of time dependent subgrid scales, also known as dynamic subscales ensures stability and more consistency of the resultant method while spatial grid size ($h$) is chosen comparatively larger than the time step ($dt$). Although a sufficient condition, $dt> Ch^2$ (for a positive constant $C$), introduced in \cite{9}, has made it possible to overcome instabilities of subgrid method with quasi-static subscales employed on some incompressible transient problem, but the condition may not be satisfied \cite{5} for anisotropic discretization of space and time. 
%More interestingly complete elimination of the subgrid sclaes has been made possible in this paper.
This paper contributes in studies of stability and convergence properties of the subgrid multiscale stabilized formulation with dynamic subscales for a significant combination of concentration dependent viscosity coefficient, also expressing shear thinning and shear thickening non-Newtonian behavior  with variable diffusion coefficients in the transport equation along with prominent numerical validation results. As per knowledge there is no such mathematical analysis based study for transportation of non-Newtonian fluid flow problem using a stabilized method available in literature.\vspace{1mm}\\
In this paper the stabilized finite element solution is shown bounded by the given data with respect to that standard full-norm based on which optimal order of convergence has been reached for implicit time discretization scheme in the discrete time setting. This uniformity helps in carrying out effective numerical validation of the theoretically established results. Both the shear thinning and shear thickening properties of the viscosity are considered and estimated separately during the derivations of both apriori and aposteriori error estimates. Standard expressions for stabilization parameters have been considered from literature \cite{10},\cite{15}. In numerical experiment part  we have verified the theoretically evaluated rate of convergence result as well as compared the performance of time dependent $ASGS$ method with three other methods which include the standard $Galerkin$ approach, the $GLS$ method and the $ASGS$ method with quasi-static subscales. All these methods, introduced earlier for different studies ($Galerkin$ method for power law model \cite{7}, $GLS$ method for Navier-Stokes equations \cite{16}, quasi-static $ASGS$ method for power law model \cite{3}) extended here for this coupled system for numerical validation purpose. We have carried out this comparison study for different combinations of power-law indices indicating shear thinning, Newtonian and shear thickening cases and concentration dependency of viscosity coefficient to cover almost all possible aspects of the coupled system. Whereas for weakly coupled system of equations (when viscosity is independent of concentration of the solute mass) time dependent $ASGS$ method has come out as the best performing numerical method among the stabilized finite element methods in approximating all the variables even for high $Re$ such as 50000, for strong coupling too this method admits optimal order of convergence. \vspace{1mm}\\
This paper is organized in this way: section 2 introduces the model problem and its variational formulation. Section 3 carries out the derivation of the stabilized multiscale finite element formulation and its stability analysis. Next section finds out apriori and aposteriori error estimates and the last section presents numerical results.

\section{Continuous problem}
\subsection{The model problem}
This section starts with presenting unsteady non-Newtonian fluid flow model followed by transient advection-diffusion-reaction equation with variable coefficients. Let $\Omega \subset$ $\mathbb{R}^d$ ($d$=2,3) be an open flow domain bounded by the boundary $\partial \Omega$. For simplifying the theoretical analysis we have particularly considered two dimensional model and it's extension to study three dimensional problem is straight forward. \vspace{1mm}\\
Now the governing equation of fluid flow phenomena  \cite{39} is to find velocities $\textbf{u}=(u_1,u_2)$: $\Omega \times (0,T)$ $\rightarrow \mathbb{R}^2$ and pressure $p$: $\Omega \times (0,T)$ $\rightarrow \mathbb{R}$ such that

\begin{equation}
\begin{split}
\rho(\frac{\partial u_1}{\partial t}+ u_1 \frac{\partial u_1}{\partial x}+ u_2 \frac{\partial u_1}{\partial y}) & =- \frac{\partial p}{\partial x}+ \frac{\partial \tau_{11}}{\partial x} + \frac{\partial \tau_{12}}{\partial y} +f_1 \hspace{2mm} in \hspace{2mm} \Omega \times (0,T)\\
\rho(\frac{\partial u_2}{\partial t}+ u_1 \frac{\partial u_2}{\partial x}+ u_2 \frac{\partial u_2}{\partial y}) & =- \frac{\partial p}{\partial y}+ \frac{\partial \tau_{12}}{\partial x} + \frac{\partial \tau_{22}}{\partial y} +f_2 \hspace{2mm} in \hspace{2mm} \Omega \times (0,T)\\
\frac{\partial u_1}{\partial x} + \frac{\partial u_2}{\partial y} & =0 \hspace{2mm} in \hspace{2mm} \Omega \times (0,T)
\end{split}
\end{equation}
where $\rho$ is density of the fluid, $f_1,f_2$ are the body forces and the components of the shear stress tensor obey the Ostwald-De Waele (Power-law) model as follows:
\begin{equation}
\tau_{ij}= 2 \eta D_{ij}= \eta (\frac{\partial u_i}{\partial x_j}+\frac{\partial u_j}{\partial x_i}) \hspace{2mm} for \hspace{2mm} i,j=1,2
\end{equation}
and $(x_1,x_2)$ indicates $(x,y)$ in this notation. Here $D_{ij}$ denotes component of deformation-rate for each $i,j=1,2$ and $\eta$ is the viscosity which behaves for non-Newtonian power-law model \cite{21} in the following form
\begin{equation}
\eta = \eta(c,\textbf{u})= K e^{Bc} \{ 2 (\frac{\partial u_1}{\partial x})^2 + 2 (\frac{\partial u_2}{\partial y})^2 + (\frac{\partial u_1}{\partial y}+ \frac{\partial u_2}{\partial x})^2 \}^{\frac{m-1}{2}}
\end{equation}
where $K$ is the consistency factor, $B$ is a dimension less constant, $m$ is the power-law index, $\textbf{u}$=$(u_1,u_2)$ is the velocity vector and $c$ denotes the concentration of the solute, transportation of which is modeled through the following $VADR$ equation.
\begin{equation}
\frac{\partial c}{\partial t}- \bigtriangledown \cdot \tilde{\bigtriangledown} c + \textbf{u} \cdot \bigtriangledown c + \alpha c  = g \hspace{2mm} in \hspace{2mm} \Omega \times (0,T) \\
\end{equation}
where the notation $\tilde{\bigtriangledown}: = (D_1 \frac{\partial}{\partial x}, D_2 \frac{\partial}{\partial y})$ \\
$D_1,D_2$ are the variable coefficients of diffusion, $\alpha$ is reaction coefficient and $g$ is the source of solute mass. \vspace{1mm}\\
Let us consider homogeneous Dirichlet boundary conditions for both (1) and (4) and the initial conditions are respectively
\begin{equation}
\begin{split}
\textbf{u}= \textbf{0}, c & =0 \hspace{2mm} on \hspace{2mm} \partial \Omega \times (0,T)\\
\textbf{u}= \textbf{u}_0, c & = c_0 \hspace{2mm} at \hspace{2mm} t=0
\end{split}
\end{equation}
In a compact form we can express the above coupled system in the following operator form
\begin{equation}
M \partial_t \textbf{U} + \mathcal{L}(\textbf{u}, \eta(c,\textbf{u}); \textbf{U})= \textbf{F}
\end{equation}
where $\textbf{U}$ denotes the triplet $(\textbf{u},p,c)$ and $M$, a matrix= $diag(\rho,\rho,0,1)$, $\partial_t \textbf{U}=[\frac{\partial \textbf{u}}{\partial t},\frac{\partial p}{\partial t} \frac{\partial c}{\partial t}]^T$ and $\textbf{F}=[f_1,f_2,0,g]^T$\\

\[
\mathcal{L} (\textbf{u}, \eta(c, \textbf{u}); \textbf{U})=
  \begin{bmatrix}
  \rho(u_1 \frac{\partial u_1}{\partial x}+ u_2 \frac{\partial u_1}{\partial y}) + \frac{\partial p}{\partial x} - \{ \frac{\partial}{\partial x} (2 \eta \frac{\partial u_1}{\partial x})+  \frac{\partial}{\partial y} ( \eta (\frac{\partial u_2}{\partial x} + \frac{\partial u_1}{\partial y})\} \\
  \rho(u_1 \frac{\partial u_2}{\partial x}+ u_2 \frac{\partial u_2}{\partial y}) + \frac{\partial p}{\partial y} - \{ \frac{\partial}{\partial x} ( \eta (\frac{\partial u_2}{\partial x} + \frac{\partial u_1}{\partial y})+  \frac{\partial}{\partial y} (2 \eta \frac{\partial u_2}{\partial y})\} \\
    \bigtriangledown \cdot \textbf{u} \\
    - \bigtriangledown \cdot \tilde{\bigtriangledown} c + \textbf{u} \cdot \bigtriangledown c + \alpha c 
  \end{bmatrix}
\]
Now adjoint $\mathcal{L}^*$ of $\mathcal{L}$ is
\[
\mathcal{L}^* (\textbf{u}, \eta(c, \textbf{u}); \textbf{U})=
  \begin{bmatrix}
  -\rho(u_1 \frac{\partial u_1}{\partial x}+ u_2 \frac{\partial u_1}{\partial y}) - \frac{\partial p}{\partial x} - \{ \frac{\partial}{\partial x} (2 \eta \frac{\partial u_1}{\partial x})+  \frac{\partial}{\partial y} ( \eta (\frac{\partial u_2}{\partial x} + \frac{\partial u_1}{\partial y})\} \\
  -\rho(u_1 \frac{\partial u_2}{\partial x}+ u_2 \frac{\partial u_2}{\partial y}) - \frac{\partial p}{\partial y} - \{ \frac{\partial}{\partial x} ( \eta (\frac{\partial u_2}{\partial x} + \frac{\partial u_1}{\partial y})+  \frac{\partial}{\partial y} (2 \eta \frac{\partial u_2}{\partial y})\} \\
    -\bigtriangledown \cdot \textbf{u} \\
    - \bigtriangledown \cdot \tilde{\bigtriangledown} c - \textbf{u} \cdot \bigtriangledown c + \alpha c 
  \end{bmatrix}
\]
Now we introduce few suitable assumptions on coefficients of the coupled system in the following:\vspace{1mm}\\
\textbf{(i)} $D_1= D_1((x,y),t) \in C^0(\mathbb{R}^2 \times (0,T);\mathbb{R})$ and $D_2= D_2((x,y),t) \in C^0(\mathbb{R}^2 \times (0,T);\mathbb{R})$ where $ C^0(\mathbb{R}^2\times (0,T);\mathbb{R})$ is the space of real valued continuous function defined on $\mathbb{R}^2$ for fixed $t \in (0,T)$. Both are bounded quantity that is we can find lower and upper bounds for both of them. \vspace{2mm} \\
\textbf{(ii)} $\rho$ and $\alpha$ are positive constants. \vspace{2mm}\\
\textbf{(iii)} The spaces of continuous solution $(\textbf{u},p,c)$ are assumed as: \\ $\textbf{u} \in L^{\infty}(0,T; (H^2(\Omega))^2)\bigcap C^{0}(0,T; (H_0^1(\Omega))^2)$ and \\
$p \in L^{\infty}(0,T;H^1(\Omega))\bigcap C^{0}(0,T;L^2_0(\Omega)) $,
 $c \in L^{\infty}(0,T;H^2(\Omega))\bigcap C^0(0,T;H^1_0(\Omega))$ \vspace{1mm}\\
\textbf{(iv)} Additional assumptions on exact velocity and concentration solutions are $u_1$, $u_2$, $c$ and all the first order derivatives of velocities i.e. $\frac{\partial u_1}{\partial x}$, $\frac{\partial u_1}{\partial y}$, $\frac{\partial u_2}{\partial x}$, $\frac{\partial u_2}{\partial y}$ are taken to be bounded functions on $\Omega$ for each $t \in (0,T)$ as well as these time derivatives, $\textbf{u}_{tt}(t),\textbf{u}_{ttt}(t), c_{tt}(t),c_{ttt}(t)$ are assumed to be bounded for a.e. $t \in J=(0,T)$.

\subsection{Variational formulation}
Assuming $f_1,f_2,g \in L^2(0,T; L^2(\Omega))$ the appropriate spaces to derive the weak formulation are $V=H^1_0(\Omega) $ and $Q=L^2_0(\Omega)$. Now denoting the product space $V^2 \times Q \times V$ by $ \bar{\textbf{V}}$ the variational formulation of (6) is to find  \textbf{U}(t)= (\textbf{u}(t),p(t),c(t)) $ \in \bar{\textbf{V}}$ such that $\forall$ \textbf{V}=(\textbf{v},q,d) $\in  \bar{\textbf{V}}$ and for $a.e.$ $t \in J$
\begin{equation}
\begin{split}
(M\partial_t \textbf{U},\textbf{V}) + B(\textbf{u},\eta(c,\textbf{u}); \textbf{U}, \textbf{V}) & = L(\textbf{V})   \hspace{2 mm} \forall \textbf{V} \in \bar{\textbf{V}}
\end{split}
\end{equation} 
where $(M \partial_t \textbf{U}, \textbf{V})= \rho \int_{\Omega} \frac{\partial u_1}{\partial t} v_1 + \rho \int_{\Omega} \frac{\partial u_2}{\partial t} v_2 +\int_{\Omega} \frac{\partial c}{\partial t} d$ \vspace{1mm}\\
$B(\textbf{u},\eta(c,\textbf{u}); \textbf{U}, \textbf{V})$= $c(\textbf{u},\textbf{u},\textbf{v}) + a_{PL}(\eta(c,\textbf{u}); \textbf{u},\textbf{v})- b(\textbf{v},p)
+ b(\textbf{u},q)+ a_{LT}(c,d)+ a_{NLT}(\textbf{u},c,d)$ and  the linear functional $L(\textbf{V})= l_{PL}(\textbf{v})+ l_{LT}(d)$.\vspace{1mm}\\
where the notations are defined in the following:\\
$c(\textbf{u},\textbf{v},\textbf{w})=\rho \int_{\Omega} ((\textbf{u} \cdot \bigtriangledown)\textbf{v})\cdot \textbf{w}+ \frac{\rho}{2} \int_{\Omega} (\bigtriangledown \cdot \textbf{u})\textbf{v} \cdot \textbf{w} $ and $b(\textbf{v},q)= \int_{\Omega} (\bigtriangledown \cdot \textbf{v}) q$\vspace{1 mm}\\
$a_{PL}(\eta(c,\textbf{u}); \textbf{u},\textbf{v})$= $\int_{\Omega}K e^{Bc} \{ 2 (\frac{\partial u_1}{\partial x})^2 + 2 (\frac{\partial u_2}{\partial y})^2 + (\frac{\partial u_1}{\partial y}+ \frac{\partial u_2}{\partial x})^2 \}^{\frac{m-1}{2}}(2 \frac{\partial u_1}{\partial x} \frac{\partial v_1}{\partial x} + \frac{\partial u_2}{\partial x} \frac{\partial v_1}{\partial y}+ \frac{\partial u_1}{\partial y} \frac{\partial v_2}{\partial x}+ \frac{\partial u_1}{\partial y} \frac{\partial v_1}{\partial y}+ \frac{\partial u_2}{\partial x} \frac{\partial v_2}{\partial x}+ 2 \frac{\partial u_2}{\partial y} \frac{\partial v_2}{\partial y})$ \vspace{1 mm}\\
 $a_{LT}(c,d) = \int_{\Omega} \tilde{\bigtriangledown}c \cdot \bigtriangledown d  + \alpha\int_{\Omega}cd $ and
 $a_{NLT}(\textbf{u},c,d)= \int_{\Omega} d \textbf{u} \cdot \bigtriangledown c$\vspace{1 mm} \\
 $l_{PL} (\textbf{v})= \int_{\Omega} (f_1 v_1+ f_2 v_2)$ and  $l_{LT}(d)= \int_{\Omega} gd$ \vspace{1 mm} \\
The addition of the incompressibility condition into the trilinear term $c(\cdot,\cdot,\cdot)$ makes it equivalent to it's original form obtained from non-linear convective term in (1). This modification provides the following important properties of $c(\cdot,\cdot,\cdot)$: \vspace{1mm}\\
(\textbf{a}) for each $\textbf{u} \in V \times V$, $c(\textbf{u},\textbf{v},\textbf{v})=0$ \hspace{1mm} $\forall$ $\textbf{v} \in V \times V$ \vspace{1mm}\\
(\textbf{b}) for \textbf{u}, \textbf{v}, \textbf{w} $\in V \times V$
\begin{equation}
  c(\textbf{u},\textbf{v},\textbf{w})\leq \begin{cases}
    C \|\textbf{u}\|_1 \|\textbf{v}\|_1 \|\textbf{w}\|_1 & \\
    C \|\textbf{u}\|_0 \|\textbf{v}\|_2 \|\textbf{w}\|_1 & \\
    C \|\textbf{u}\|_2 \|\textbf{v}\|_1 \|\textbf{w}\|_0 & \\
    C \|\textbf{u}\|_0 \|\textbf{v}\|_1 \|\textbf{w}\|_{L^{\infty}(\Omega)} 
  \end{cases}
\end{equation}
where $C$ is a constant and $\| \cdot \|_i$ for i=0,1,2 denote the standard $L^2, H^1, H^2$ full norms respectively.  From now onward for simplicity we use $\| \cdot \|$ instead of $\| \cdot \|_0$ to denote $L^2(\Omega)$ norm.

\section{Discrete formulation}
\subsection{Space and time discretizations}
Let the domain $\Omega$ be discretized into $n_{el}$ numbers of subdomains $\Omega_k$ for k=1,2,..., $n_{el}$. Let $h_k$ be the diameter of each subdomain $\Omega_k$ with $h$= $\underset{k=1,2,...n_{el}}{max} h_k$.
%Let $\tilde{\Omega}= \bigcup_{k=1}^{n_{el}} \Omega_k$ be the union of interior elements.\vspace{1 mm}\\
Let the  finite dimensional spaces  $V_h \subset V$ and $Q_h \subset Q$ be considered as: \vspace{1mm}\\
$V_h= \{ v \in V: v(\Omega_k)= \mathcal{P}^l(\Omega_k)\} $ and 
$Q_h= \{ q \in Q : q(\Omega_k)= \mathcal{P}^{l-1}(\Omega_k)\}$ \vspace{1 mm}\\
where  $\mathcal{P}^l(\Omega_k)$ denotes complete polynomial upto order $l$ over each $\Omega_k$ for k=1,2,..., $n_{el}$ along with this inclusion assumption $\bigtriangledown \cdot V_h \subset Q_h$. Considering a likewise notation $\bar{\textbf{V}}_h$ for denoting the product space $V_h \times V_h \times Q_h \times V_h$, the standard \textbf{Galerkin finite element formulation} for the variational form (7) is to find $\textbf{U}_h(t) $= $(\textbf{u}_h(t),p_h(t),c_h(t))$ $ \in \bar{\textbf{V}}_h$ such that $\forall$ $\textbf{V}_h=(\textbf{v}_h,q_h,d_h)$ $\in \bar{\textbf{V}}_h$ and for $a.e.$ $t \in J$
\begin{equation}
(M\partial_t \textbf{U}_h,\textbf{V}_h) + B(\textbf{u}_h, \eta(c_h,\textbf{u}_h); \textbf{U}_h, \textbf{V}_h) = L(\textbf{V}_h)   
\end{equation}
where $(M\partial_t \textbf{U}_h,\textbf{V}_h)$= $ \rho (\frac{\partial u_{1h}}{\partial t}, v_{1h})+ \rho (\frac{\partial u_{2h}}{\partial t}, v_{2h})+ (\frac{\partial c_h}{\partial t}, d_h)$ \vspace{2mm}\\
 $B(\textbf{u}_h, \eta(c_h,\textbf{u}_h); \textbf{U}_h, \textbf{V}_h)$ = $ c(\textbf{u}_h,\textbf{u}_h,\textbf{v}_h)+ a_{PL}(\eta(c_h, \textbf{u}_h);\textbf{u}_h,\textbf{v}_h)- b(\textbf{v}_h,p_h)+ b(\textbf{u}_h,q_h)+ a_{LT}(c_h,d_h)+ a_{NLT}(\textbf{u}_h,c_h,d_h)$ \vspace{1mm}\\
and  $L(\textbf{V}_h)= l_{PL}(\textbf{v}_h)+ l_{LT}(d_h) $ \\
In addition let us consider the initial conditions $(\textbf{u}_h, \textbf{v}_h)\mid_{t=0}=(\textbf{u}_0, \textbf{v}_h)$ $\forall \textbf{v}_h \in V_h \times V_h $ and
  $(c_h,d_h)\mid_{t=0}= (c_0,d_h)$ $\forall d_h \in V_h$. \vspace{1mm}\\
For introducing time discretization we mention here few notations: let us divide the time $T$ into $N$ number of steps and the $n^{th}$ time step is $t_n$= $ndt$ where $dt=\frac{T}{N}$. Now for given $\theta \in [0,1]$ and any function $f: \Omega \times (0,T)$ $\rightarrow \mathbb{R}$, it's approximations are of the following form
\begin{equation}
\begin{split}
f^n & = f(\cdot , t_n) \hspace{4 mm} for \hspace{2 mm} 0 \leq n \leq N\\
f^{n,\theta} &= \frac{1}{2} (1 + \theta) f^{(n+1)} + \frac{1}{2} (1- \theta) f^n \hspace{4mm} for \hspace{2mm} 0\leq n \leq N-1
\end{split}
\end{equation}
This discretization formula becomes $Crank$-$Nicolson$ (second order convergent) for $\theta=0$ and $backward$ $Euler$ (first order convergent) for $\theta=1$. \vspace{1mm}\\
Let $\textbf{u}^{n,\theta},p^{n,\theta},c^{n,\theta}$ be approximations of $\textbf{u}(\textbf{x},t^{n,\theta}), p(\textbf{x},t^{n,\theta}),c(\textbf{x},t^{n,\theta})$ respectively. Now by Taylor series expansion \cite{34},we have 
\begin{equation}
\begin{split}
\frac{u_i^{n+1}-u_i^n}{dt} & = \frac{\partial u_i}{\partial t}(\textbf{x},t^{n,\theta}) + TE_i\mid_{t=t^{n,\theta}} \hspace{5mm} \forall \textbf{x} \in \Omega \hspace{2mm} i=1,2 \\
\frac{c^{n+1}-c^n}{dt} & = \frac{\partial c}{\partial t}(\textbf{x},t^{n,\theta}) + TE_3\mid_{t=t^{n,\theta}} \hspace{5mm} \forall \textbf{x} \in \Omega
\end{split}
\end{equation}
where the truncation errors $TE_i\mid_{t=t^{n,\theta}}$ $\simeq$ $TE_i^{n,\theta}$, (for i=1,2,3) depend upon the time-derivatives of the respective variables and the time step $dt$ in the following way \cite{34}:
\begin{equation}
\begin{split}
\|TE_1^{n,\theta}\| & \leq
      \begin{cases}
      C' dt \|u_{1,tt}^{n,\theta}\|_{L^{\infty}(t^n,t^{n+1},L^2)} & if \hspace{1mm} \theta=1 \\
    C'' dt^2 \|u_{1,tt}^{n,\theta}\|_{L^{\infty}(t^n,t^{n+1},L^2)} & if \hspace{1mm} \theta=0
      \end{cases}
\end{split}
\end{equation}
%The truncation error is of $O(dt \theta +dt^2(1-\theta)^3+dt^2(1+\theta)^3)$ \cite{RefQ}\\
The above relation holds for $TE_2$ and $TE_3$ in similar manner.
Now in particular for backward Euler scheme ($\theta$=1) applying assumption $\textbf{(iv)}$ we have another property as follows:
\begin{equation}
\begin{split}
\|TE_1^{n,\theta}\| & \leq C' dt  \|u_{1,tt}^{n+1}\|_{L^{\infty}(t^n,t^{n+1},L^2)} \leq \tilde{C}' dt
\end{split}
\end{equation}
Similarly
\begin{equation}
\begin{split}
\|TE_2^{n,\theta}\| & \leq \tilde{C}' dt \hspace{2mm} and \hspace{2mm}
\|TE_3^{n,\theta}\|  \leq \tilde{C}' dt \\
\end{split}
\end{equation}
After introducing all the required definitions the fully discrete Galerkin finite element formulation is to find $\textbf{U}_h^{n+1}= (\textbf{u}_h^{n+1},p_h^{n+1},c_h^{n+1}) \in \bar{\textbf{V}}_h$ for given $\textbf{U}_h^n = (\textbf{u}_h^n,p_h^n,c_h^n)\in \bar{\textbf{V}}_h$ such that $\forall \hspace{1mm} \textbf{V}_h=(\textbf{v}_h,q_h,d_h) \in \bar{\textbf{V}}_h $
\begin{equation}
(M\frac{(\textbf{U}_h^{n+1}-\textbf{U}_h^n)}{dt}, \textbf{V}_h)+ B(\textbf{u}_h^{n}, \eta(c_h^{n},\textbf{u}_h^{n} ) ;\textbf{U}_h^{n,\theta}, \textbf{V}_h) = L(\textbf{V}_h) + (\textbf{TE}^{n,\theta},\textbf{V}_h)  
\end{equation}
%where $\epsilon$ takes only two values 0 and 1. For $\epsilon=0$ these formulations become linear, otherwise non-linear.  

\subsection{ Multiscale stabilized finite element formulation}
It is well known that the Galerkin formulation (9) suffers from lack of numerical stability in the case of convection dominated flows and while the velocity-pressure spaces fail to satisfy the $inf$-$sup$ compatibility condition. The $subgrid$ $multiscale$ stabilized formulation has been introduced to overcome these instabilities. This method starts with decomposition of the continuous solution $\textbf{U}$ into the known finite element solution $\textbf{U}_h$ and an unknown subgrid scale component $\tilde{\textbf{U}}$ such that $\textbf{U}=\textbf{U}_h + \tilde{\textbf{U}}$ where $\tilde{\textbf{U}} \in \tilde{\textbf{V}}$, a subspace of $\textbf{V}$ completing $\textbf{V}_h$ in $\textbf{V}$. Applying alike decomposition on the test functions, the weak formulation (7) can be equivalently expressed as
\begin{equation}
\begin{split}
(M \partial_t \textbf{U}_h+M \partial_t \tilde{\textbf{U}},\textbf{V}_h)+B(\textbf{u}_h, \eta(c_h,\textbf{u}_h); \textbf{U}_h+ \tilde{\textbf{U}}, \textbf{V}_h) & = L(\textbf{V}_h) \\
(M \partial_t \textbf{U}_h+M \partial_t \tilde{\textbf{U}},\tilde{\textbf{V}})+B(\textbf{u}_h, \eta(c_h,\textbf{u}_h); \textbf{U}_h +  \tilde{\textbf{U}}, \tilde{\textbf{V}}) & = L(\tilde{\textbf{V}}) \\
\end{split}
\end{equation}
Now the aim is to express the unknown subgrid scale $\tilde{\textbf{U}}$ in terms of $\textbf{U}_h$. For this purpose on integrating the second sub-equation and approximating the differential operator $\mathcal{L}$ by an algebraic operator $\tau_k$ we have over each sub domain $\Omega_k$
\begin{equation}
M \partial_t \tilde{\textbf{U}} + \tau_k^{-1} \tilde{\textbf{U}} =\textbf{R}:= \textbf{F}-M \partial_t \textbf{U}_h- \mathcal{L}(\textbf{u}_h, \eta(c_h, \textbf{u}_h); \textbf{U}_h)
\end{equation}
where $\textbf{R}$ is the residual vector. Applying time discretization rule, in particular backward Euler scheme on the above equation we have
\begin{equation}
\tilde{\textbf{U}} = \tau_k'(\textbf{R}+ \textbf{d})
\end{equation}
where the stabilization parameter $\tau_k$ = $diag(\tau_{1k},\tau_{1k},\tau_{2k}, \tau_{3k})$ and $\tau_k'$= $(\frac{1}{dt}M+ \tau_k^{-1})^{-1}$= $diag(\frac{\tau_{1k} dt}{dt+ \rho \tau_{1k}}, \frac{\tau_{1k} dt}{dt+ \rho \tau_{1k}}, \tau_{2k}, \frac{\tau_{3k} dt}{dt+  \tau_{3k}})= diag(\tau_{1k}', \tau_{1k}', \tau_{2k}', \tau_{3k}')$ (say) and $\textbf{d}$= $\sum_{i=1}^{n+1}(\frac{1}{dt}M\tau_k')^i(\textbf{F} -M\partial_t \textbf{U}_h - \mathcal{L}(\textbf{u}_h, \eta(c_h, \textbf{u}_h) ;\textbf{U}_h))$.\vspace{1mm}\\
Substituting all these expressions of $\tilde{\textbf{U}}$ into the first sub equation of (16) we finally have arrived at the \textbf{stabilized algebraic subgrid multiscale} ($ASGS$) finite element formulation in the following: to 
find $\textbf{U}_h(t) $= $(\textbf{u}_h(t),p_h(t),c_h(t))$ $ \in \bar{\textbf{V}}_h$ such that $\forall$ $\textbf{V}_h=(\textbf{v}_h,q_h,d_h)$ $\in \bar{\textbf{V}}_h$ and for $a.e.$ $t \in J$
\begin{equation}
(M\partial_t \textbf{U}_h,\textbf{V}_h) + B_{ASGS}(\textbf{u}_h, \eta(c_h,\textbf{u}_h) ; \textbf{U}_h, \textbf{V}_h)  = L_{ASGS}(\textbf{V}_h)  
\end{equation}
where $B_{ASGS}(\textbf{u}_h,  \eta(c_h,\textbf{u}_h) ;\textbf{U}_h, \textbf{V}_h)= B(\textbf{u}_h,  \eta(c_h,\textbf{u}_h); \textbf{U}_h, \textbf{V}_h)+ \sum_{k=1}^{n_{el}} (\tau_k'(M \\
\partial_t \textbf{U}_h + \mathcal{L}(\textbf{u}_h, \eta(c_h,\textbf{u}_h) ;\textbf{U}_h)-\textbf{d}), -\mathcal{L}^*(\textbf{u}_h, \eta(c_h,\textbf{u}_h);\textbf{V}_h))_{\Omega_k}- \sum_{k=1}^{n_{el}}((I-\tau_k^{-1}\tau_k') \\(M\partial_t \textbf{U}_h + \mathcal{L}(\textbf{u}_h, \eta(c_h,\textbf{u}_h);\textbf{U}_h)), \textbf{V}_h)_{\Omega_k} 
-\sum_{k=1}^{n_{el}} (\tau_k^{-1}\tau_k' \textbf{d}, \textbf{V}_h)_{\Omega_k}$ \vspace{2 mm}\\
$L_{ASGS}(\textbf{V}_h)= L(\textbf{V}_h)+ \sum_{k=1}^{n_{el}}(\tau_k' \textbf{F}, -\mathcal{L}^*(\textbf{u}_h, \eta(c_h,\textbf{u}_h);\textbf{V}_h))_{\Omega_k}- \sum_{k=1}^{n_{el}}((I-\tau_k^{-1}\tau_k')\textbf{F}, \textbf{V}_h)_{\Omega_k}$  \vspace{1 mm} \\
Each of the stabilization parameters \cite{3}, \cite{10} is of the following form:
\begin{equation}
\begin{split}
\tau_{1k} &= \tau_{1}= (c_1 \frac{\eta_0}{h^2}+  c_2 \frac{\rho \mid \textbf{u}_h \mid}{h})^{-1} \\
\tau_{2k} &=\tau_{2}=\frac{h^2}{c_1 \tau_{1}} \\
\tau_{3k} & = \tau_{3}= c_3(\frac{9D_m}{4h^2} + \frac{3\mid \textbf{u}_h \mid}{2h} + \alpha )^{-1}
\end{split}
\end{equation}
where $c_1,c_2,c_3$ are suitably chosen positive parameters, $\eta_0$ is total viscosity and $\mid \textbf{u}_h \mid $ of the computed velocity and $D_m$ is deduced in \cite{10}. 

\subsection{Stability analysis}
In this section we discuss about stability analysis of an equivalent system of equations to the stabilized formulation (19). From the above derivation it can be easily observed that (19) is equivalent to the following system of equations: for given $\textbf{U}_h^{n}= (\textbf{u}_h^{n},p_h^{n},c_h^{n}) \in \bar{\textbf{V}}_h$ and $\tilde{\textbf{U}}^n=(\tilde{\textbf{u}}^n, \tilde{p}^n,\tilde{c}^n) \in \tilde{\textbf{V}}$ find $\textbf{U}_h^{n+1}= (\textbf{u}_h^{n+1},p_h^{n+1},c_h^{n+1}) \in \bar{\textbf{V}}_h$ and $\tilde{\textbf{U}}^{n+1}=(\tilde{\textbf{u}}^{n+1}, \tilde{p}^{n+1},\tilde{c}^{n+1}) \in \tilde{\textbf{V}}$ such that $\forall \hspace{1mm} \textbf{V}_h=(\textbf{v}_h,q_h,d_h) \in \bar{\textbf{V}}_h $
\begin{multline}
(M \frac{\textbf{U}_h^{n+1}-\textbf{U}_h^n}{dt}, \textbf{V}_h)+B(\textbf{u}_h^n,\eta_h^n; \textbf{U}_h^{n+1}, \textbf{V}_h)- (\tilde{\textbf{U}}^{n+1}, M \partial_t \textbf{V}_h)+ (\tilde{\textbf{U}}^{n+1},\mathcal{L}(\textbf{u}_h,\eta_h;\textbf{V}_h))  = L(\textbf{V}_h)\\
M \frac{\tilde{\textbf{U}}^{n+1}-\tilde{\textbf{U}}^n}{dt} + \tau_k^{-1}\tilde{\textbf{U}}^{n+1}  = \textbf{F}- M \partial_t \textbf{U}_h^{n+1}- \mathcal{L}(\textbf{u}_h^n,\eta_h^n;\textbf{U}_h^{n+1})
\end{multline}
where the notation $\eta_h$ is an abbreviation of $\eta(c_h,\textbf{u}_h)$. For time discretization here we have particularly worked with backward Euler scheme instead of the general approach mentioned in section 3.1.
Let us mention here few assumptions needed in order to establish the stability: assuming $\textbf{u}_0,c_0 \in L^2(\Omega)$ ensures that $\|\textbf{u}_h^0\|$ and $\|c_h^0\|$ are uniformly bounded and we assume $\tilde{\textbf{u}}|_{t=0}=0$ and $\tilde{c}|_{t=0}=0$. Again the assumptions  $f_1,f_2,g \in L^2(0,T; L^2(\Omega))$ implies that for fully discrete problem the corresponding force and source terms $\{f_1^n\}, \{f_2^n\}, \{g^n\} \in l^2(L^2(\Omega))$.
\begin{theorem}
For regular partitions satisfying inverse inequalities and $(\textbf{u}_h^{n+1},p_h^{n+1},c_h^{n+1})$ and $(\tilde{\textbf{u}}^{n+1}, \tilde{p}^{n+1},\tilde{c}^{n+1})$ being solutions of (21) the following stability bounds hold for all $dt>0$
\begin{multline}
\underset{n=1,...,N-1}{max} \{ \|\textbf{u}^{n+1}_h\|^2 +\|c^{n+1}_h\|^2+ \| \tilde{\textbf{u}}^{n+1}\|^2+\|\tilde{c}^{n+1}\|^2 \} +\\
 \sum_{n=1}^{N-1}dt (C_1^s \|\nabla \textbf{u}_h^{n+1}\|^2 + C_2^s  \|\textbf{u}_h^{n+1}\|^2 + C_3^s \| \tau_{1k}^{-\frac{1}{2}} \tilde{\textbf{u}}^{n+1}\|^2)+ \\
 \sum_{n=1}^{N-1}dt (C_4^s \|\nabla c_h^{n+1}\|^2 + C_5^s  \|c_h^{n+1}\|^2 + C_6^s \| \tau_{3k}^{-\frac{1}{2}} \tilde{c}^{n+1}\|^2)\\
  \leq C_7^s \{ \sum_{n=1}^{N-1} dt( \|f_1^{n+1}\|^2 +\|f_2^{n+1}\|^2+\|c^{n+1}\|^2)+  \|\textbf{u}_h^{0}\|^2+\|c_h^0\|^2\}
\end{multline}
where $C_i^s$ are positive constants for i=1,...,7. Furthermore $\{f_1^n\}, \{f_2^n\}, \{g^n\} \in l^2(L^2(\Omega))$ and  uniformly bounded $\|\textbf{u}_h^0\|$ and $\|c_h^0\|$ implies that
\begin{center}
$\{\textbf{u}^n_h\},  \{c^n_h\} \in l^{\infty}(L^2(\Omega)) \bigcap l^2(H^1(\Omega)) $\\
$ \{\tilde{\textbf{u} }^n\}, \{\tilde{c}^n\} \in l^{\infty}(L^2(\Omega)), \{ \tau_{1k}^{-\frac{1}{2}} \tilde{\textbf{u}}^{n+1} \}, \{ \tau_{3k}^{-\frac{1}{2}} \tilde{c}^{n+1} \} \in  l^2(L^2(\Omega))$
\end{center}
\end{theorem}
\begin{proof}
Substituting $\textbf{V}_h$ by $(\textbf{u}_h^{n+1},p_h^{n+1},c_h^{n+1})$ in first sub equation of (21) and integrating the second one after multiplying it by $\tilde{\textbf{U}}^{n+1}$ on both sides we have
\begin{multline}
\frac{M}{dt}( \|\textbf{U}_h^{n+1}\|^2+ \|\tilde{\textbf{U}}^{n+1}\|^2)-(\frac{M}{dt} (\textbf{U}^n_h +  \tilde{\textbf{U}}^n), \textbf{U}_h^{n+1})+ \tau_k^{-1} \|\tilde{\textbf{U}}^{n+1}\|^2+2 \alpha \|c^{n+1}\|^2 \\
B(\textbf{u}_h^n,\eta_h^n; \textbf{U}_h^{n+1}, \textbf{U}_h^{n+1})-2 \sum_{k=1}^{n_{el}}(\tilde{u}_1^{n+1}, \frac{\partial}{\partial x}(2 \eta_h \frac{\partial u_{1h}^{n+1}}{\partial x})+\frac{\partial}{\partial y}(\eta_h( \frac{\partial u_{2h}^{n+1}}{\partial x}+ \frac{\partial u_{1h}^{n+1}}{\partial y})))_k-2 \sum_{k=1}^{n_{el}}(\tilde{u}_2^{n+1},\\
\frac{\partial}{\partial x}(\eta_h( \frac{\partial u_{2h}^{n+1}}{\partial x}+ \frac{\partial u_{1h}^{n+1}}{\partial y}))+\frac{\partial}{\partial y}(2 \eta_h \frac{\partial u_{2h}^{n+1}}{\partial y}))_k-2 \sum_{k=1}^{n_{el}}(\tilde{c}^{n+1}, \nabla \cdot \tilde{\nabla} c_h^{n+1})_k=(\textbf{F}^{n+1}, \textbf{U}_h^{n+1})+ (\textbf{F}^{n+1}, \tilde{\textbf{U}}^{n+1})
\end{multline}
Now we separately find bounds for each of the above terms. Applying the Cauchy-Schwarz and Youngs inequalities subsequently we have obtained the following results
\begin{equation}
\begin{split}
(\frac{M}{dt} (\textbf{U}^n_h +  \tilde{\textbf{U}}^n), \textbf{U}_h^{n+1}) & \leq \frac{\rho}{ dt}(\|\textbf{u}^n_h\|^2+ \|\textbf{u}_h^{n+1}\|^2+ \| \tilde{\textbf{u}}^n\|^2)\\
& \quad + \frac{1}{dt} (\|c^n_h\|^2+\|c^{n+1}_h\|^2+ \| \tilde{c}^n\|^2)\\
(\textbf{F}^{n+1}, \textbf{U}_h^{n+1}) + (\textbf{F}^{n+1}, \tilde{\textbf{U}}^{n+1}) & \leq \epsilon_1 (\|f_1^{n+1}\|^2 +\|f_2^{n+1}\|^2+\|c^{n+1}\|^2)+\\
& \quad \frac{1}{2 \epsilon_1}( \|\textbf{u}_h^{n+1}\|^2 + \| \tilde{\textbf{u}}^{n+1}\|^2+ \|c^{n+1}_h\|^2+\| \tilde{c}^n\|^2) \\
\end{split}
\end{equation}
To find an appropriate bound on $B(\cdot,\cdot; \cdot,\cdot)$ we have applied property $\textbf{(a)}$ on the trilinear term $c(\cdot,\cdot,\cdot)$ and assumptions $\textbf{(i)}-\textbf{(iv)}$ on the coefficients and finally arrived at the following
\begin{equation}
\begin{split}
B(\textbf{u}_h^n,\eta_h^n; \textbf{U}_h^{n+1}, \textbf{U}_h^{n+1}) 
& \geq 2 \eta_l  \| \nabla \textbf{u}_h^{n+1} \|^2 + D_l \| \nabla c^{n+1}\|^2 + \alpha \|c^{n+1}\|^2 \\
\end{split}
\end{equation}
Now to estimate the remaining terms which are defined over each sub domain $\Omega_k$ we make use of an important observation: by the virtue of the choice of the finite element spaces $V_h$ and $Q_h$, we can clearly say that over each element sub domain every function belonging to that spaces and their first and second order derivatives all are bounded functions.
\begin{equation}
\begin{split}
 \sum_{k=1}^{n_{el}}(\tilde{u}_1^{n+1}, \frac{\partial}{\partial x}(2 \eta_h \frac{\partial u_{1h}^{n+1}}{\partial x})+\frac{\partial}{\partial y}(\eta_h( \frac{\partial u_{2h}^{n+1}}{\partial x}+ \frac{\partial u_{1h}^{n+1}}{\partial y})))_k & \leq \bar{C}_1 \|\tilde{u}_1^{n+1}\| \\
  \sum_{k=1}^{n_{el}}(\tilde{u}_2^{n+1}, \frac{\partial}{\partial x}(\eta_h( \frac{\partial u_{2h}^{n+1}}{\partial x}+ \frac{\partial u_{1h}^{n+1}}{\partial y}))+\frac{\partial}{\partial y}(2 \eta_h \frac{\partial u_{2h}^{n+1}}{\partial y}))_k & \leq \bar{C}_2 \|\tilde{u}_2^{n+1}\| \\
\sum_{k=1}^{n_{el}}(\tilde{c}^{n+1}, \nabla \cdot \tilde{\nabla} c_h^{n+1})_k & \leq \bar{C}_3 \|\tilde{c}^{n+1}\|
\end{split}
\end{equation} 
where $\bar{C}_i$ for i=1,2,3 are positive constants obtained after applying the above observation. Now combining all these estimated results in (23), multiplying the resulting equation by $dt$ and adding for $n=1$ to $N-1$ we prove the stability result (22).
\end{proof}

\section{Error estimates}
We start this section with introducing the projection operators corresponding to each unknown variables followed by notation of error and it's component wise splitting. Later we go to derive $ apriori$ and $aposteriori$ error estimates.

\subsection{Projection operators : Error splitting}
Let us introduce the projection operator for each of these error components.\vspace{1 mm}\\
(P1) For any $\textbf{u} \in (H^2(\Omega))^2 $ we assume that there exists an interpolation $I^h_{\textbf{u}}:  (H^2(\Omega))^2 \longrightarrow  (V_h)^2 $ satisfying $b(\textbf{u}-I^h_{\textbf{u}}\textbf{u}, q_h)=0$ \hspace{2mm} $\forall q_h \in Q_h$. \vspace{2mm}\\
(P2) Let $I^h_p: H^1(\Omega) \longrightarrow Q_h$ be the $L^2$ orthogonal projection given by \\ $\int_{\Omega}(p-I^h_pp)q_h=0$  \hspace{1mm} $\forall q_h \in Q_h$ and for any $p \in H^1(\Omega)$ and\vspace{2mm}\\
(P3) Let $I^h_{c}: H^2(\Omega) \longrightarrow V_h$ be the $L^2$ orthogonal projection given by \\ $\int_{\Omega}(c-I^h_c c)d_h=0$ \hspace{1mm} $ \forall d_h \in V_h$ and for any $c \in H^2(\Omega)$. \vspace{2mm}\\
Let $\textbf{e}=(e_{\textbf{u}},e_p,e_c)$ denote the error where the components are $e_{\textbf{u}}=\sum_{i=1}^2 e_{ui} \bar{e}_{i}= \sum_{i=1}^2 (u_i-u_{ih}) \bar{e}_{i}, e_p= (p-p_h)$ and $e_c=(c-c_h)$. Now each component of the error can be split into two parts, namely interpolation part, $E^I$ and auxiliary part, $E^A$ as follows:  \vspace{1mm}\\
$e_{\textbf{u}}=(\textbf{u}-\textbf{u}_h)=(\textbf{u}-I^h_{\textbf{u}} \textbf{u})+(I^h_{\textbf{u}} \textbf{u}-\textbf{u}_h)= E^{I}_{\textbf{u}}+ E^{A}_{\textbf{u}}$ \vspace{1mm}\\
Similarly, 
$e_{p}=E^{I}_{p}+ E^{A}_{p}$, and 
$e_{c}=E^{I}_{c}+ E^{A}_{c}$ \vspace{2mm}\\
At this point let us mention the standard \textbf{interpolation estimation} result \cite{4} in the following: for any exact solution with regularity upto (r+1)
\begin{equation}
\|v-I^h_v v\|_l = \|E^I_v\|_l \leq C(p,\Omega) h^{r+1-l} \|v\|_{r+1} 
\end{equation}
where l ($\leq r+1$) is a positive integer and C is a constant depending on m and the domain. 
Now we put two results \cite{4} in the following using the properties of projection operators and these results will be used in error estimations.
\begin{result}
 For any interpolation error $E^I$
\begin{equation}
(\frac{\partial}{\partial t} E^{I}, v_h)=0 \hspace{2mm} \forall v_h \in V_h
\end{equation}
\end{result}
\begin{result}
 For any given auxiliary error $E^{A,n}$ and unknown $E^{A,n+1}$
\begin{equation}
(\frac{\partial}{\partial t} E^{A,n}, E^{A,n,\theta}) \geq  \frac{1}{2 dt} (\|E^{A,n+1}\|^2- \|E^{A,n}\|^2)
\end{equation}
\end{result}

\subsection{Apriori error estimation}
In this section we will find $apriori$ error bound, which depends on the exact solution. Here we first estimate $auxiliary$ error bound and later using that we will find $apriori$ error estimate. Before going into the detailed derivation here the definitions of norms, required for the upcoming estimates have been introduced below. For simplifying the notations we denote the spaces $L^{\infty}(0,T; L^2(\Omega)) \bigcap L^{2}(0,T; H^1(\Omega)) $ and $L^{2}(0,T; L^2(\Omega))$ by $\textbf{M}$ and $\textbf{N}$ respectively.
\begin{center}
$\|f\|_{\textbf{M}}^2  = \underset{0\leq n \leq N}{max} \|f^n\|^2 +dt \sum_{n=0}^{N-1} ( \| f^{n,\theta} \|^2  + \| \frac{\partial f}{\partial x}^{n,\theta} \|^2 + \| \frac{\partial f}{\partial y}^{n,\theta}\|^2)$\\
$\|f\|_{\textbf{N}}^2 = dt \sum_{n=0}^{N-1} \|f^{n,\theta}\|^2 dt $
\end{center}
%\begin{equation}
%\begin{split}
%\|f\|_{\textbf{M}}^2 & = \underset{0\leq n \leq N}{max} \|f^n\|^2 +dt \sum_{n=0}^{N-1} ( \| f^{n,\theta} \|^2  + \| \frac{\partial f}{\partial x}^{n,\theta} \|^2 + \| \frac{\partial f}{\partial y}^{n,\theta}\|^2)\\
%\|f\|_{\textbf{N}}^2 &= dt \sum_{n=0}^{N-1} \|f^{n,\theta}\|^2 dt 
%\end{split}
%\end{equation}

\begin{theorem} (Auxiliary error estimate) \hspace{1mm}
For sufficiently regular continuous solutions $(\textbf{u},p,c)$ satisfying the assumptions \textbf{(iii)}-\textbf{(iv)} and computed solutions $(\textbf{u}_h,p_h,c_h)$ belonging to $V_h \times V_h \times Q_h \times V_h$ satisfying (19), assume $dt$ is sufficiently small and positive, and the coefficients satisfy the assumptions \textbf{(i)}-\textbf{(ii)}. Then there exists a constant $\tilde{C}$, depending upon $\textbf{u},p,c$ such that
\begin{equation}
\|E^A_{u1}\|^2_{\textbf{M}} + \|E^A_{u2}\|^2_{\textbf{M}}+ \|E^A_p\|_{\textbf{N}}^2  + \|E^A_{c}\|^2_{\textbf{M}} \leq \tilde{C} (h^2+ dt^{2r})
\end{equation}
where
\begin{equation}
    r=
    \begin{cases}
      1, & \text{if}\ \theta=1 \\
      2, & \text{if}\ \theta=0
    \end{cases}
  \end{equation}
\end{theorem}
\begin{proof} In first part we will find bound for auxiliary error part of velocity $\textbf{u}$ and concentration $c$ with respect to $\textbf{M}$-norm and in the second part we will estimate auxiliary error for pressure term with respect to $\textbf{N}$ norm and finally combining them we will arrive at the desired result. \vspace{2mm} \\
\textbf{First part} Subtracting fully-discrete $ASGS$ stabilized formulation (19) from fully-discrete weak formulation (7) we have $\forall \hspace{1mm} \textbf{V}_h \in V_h \times V_h \times Q_h \times V_h$
\begin{multline}
(M\frac{(\textbf{U}^{n+1}-\textbf{U}^{n+1}_{h})- (\textbf{U}^{n}-\textbf{U}^{n}_{h})}{dt},\textbf{V}_{h}) + \{ B(\textbf{u}^{n},  \eta(c^{n},\textbf{u}^{n} );\textbf{U}^{n,\theta}, \textbf{V}_h) \\
-B(\textbf{u}_h^{n},  \eta(c_h^{n},\textbf{u}_h^{n} );\textbf{U}^{n,\theta}_h, \textbf{V}_h)\}
+ \sum_{k=1}^{n_{el}}(\tau_k'(M\partial_t (\textbf{U}^n-\textbf{U}^n_h)+\\
 \mathcal{L}(\textbf{u}^{n},  \eta(c^{n},\textbf{u}^{n} ) ;\textbf{U}^{n,\theta})-\mathcal{L}(\textbf{u}^{n}_h,  \eta(c_h^{n},\textbf{u}_h^{n} ) ;\textbf{U}^{n,\theta}_h)),\\
-\mathcal{L}^* (\textbf{u}_h,  \eta(c_h,\textbf{u}_h );\textbf{V}_h))_{\Omega_k} 
+\sum_{k=1}^{n_{el}}((I-\tau_k^{-1}\tau_k')(M \partial_t(\textbf{U}^{n}-\textbf{U}^{n}_h) + \\
 \mathcal{L}(\textbf{u}^{n}, \eta(c^{n},\textbf{u}^{n} ) ;\textbf{U}^{n,\theta})-\mathcal{L}(\textbf{u}^{n}_h,  \eta(c_h^{n},\textbf{u}_h^{n} ) ;\textbf{U}^{n,\theta}_h)), -\textbf{V}_h)_{\Omega_k}\\
+ \sum_{k=1}^{n_{el}} (\tau_k^{-1}\tau_k' \textbf{d}, \textbf{V}_h)_{\Omega_k}+\sum_{k=1}^{n_{el}}(\tau_k' \textbf{d},-\mathcal{L}^*(\textbf{u}_h,  \eta(c_h,\textbf{u}_h ); \textbf{V}_h))_{\Omega_k}=(\textbf{TE}^{n,\theta}, \textbf{V}_h)
\end{multline}
where $\textbf{d}$= $(\sum_{i=1}^{n+1}(\frac{1}{dt}M\tau_k')^i)(M\partial_t (\textbf{U}^n-\textbf{U}^n_h) + \mathcal{L}(\textbf{u}^{n},\eta(c^{n},\textbf{u}^{n})  ;\textbf{U}^{n,\theta})-\\
\mathcal{L}(\textbf{u}_h^{n}, \eta(c_h^{n},\textbf{u}_h^{n}) ;\textbf{U}_h^{n,\theta}))$. \vspace{2mm}\\
 Now applying error splitting for each of the terms and later using(28) and property (P1) of projection operator we have rearranged the above equation (32) as follows: $\forall  \textbf{V}_h \in V_h \times V_h \times Q_h \times V_h$
\begin{multline}
\rho (\frac{E^{A,n+1}_{\textbf{u}}-E^{A,n}_{\textbf{u}}}{dt}, \textbf{v}_h) + (\frac{E^{A,n+1}_c- E^{A,n}_c}{dt}, d_h)+ a_{LT}(E^{A,n,\theta}_c,d_h)+\\
\quad \int_{\Omega} \eta^n (2 \frac{\partial E^{A,n,\theta}_{u1}}{\partial x} \frac{\partial v_{1h}}{\partial x}+ \frac{\partial E^{A,n,\theta}_{u1}}{\partial y} \frac{\partial v_{1h}}{\partial y}+ \frac{\partial E^{A,n,\theta}_{u2}}{\partial x} \frac{\partial v_{2h}}{\partial x}+ 2 \frac{\partial E^{A,n,\theta}_{u2}}{\partial y} \frac{\partial v_{2h}}{\partial y})\\
\quad + \int_{\Omega} \eta^n \{(E^{A,n,\theta}_{u1})^2+ (E^{A,n,\theta}_{u2})^2\} \\
 = b(\textbf{v}_h, E^{A,n,\theta}_p) -b(E^{A,n,\theta}_{\textbf{u}},q_h) +b(\textbf{v}_h, E^{I,n,\theta}_p) +  \\
\quad \int_{\Omega} \eta^n \{(E^{A,n,\theta}_{u1})^2+ (E^{A,n,\theta}_{u2})^2\}- \int_{\Omega} \eta^n (\frac{\partial E^{A,n,\theta}_{u2}}{\partial x} \frac{\partial v_{1h}}{\partial y}+ \frac{\partial E^{A,n,\theta}_{u1}}{\partial y} \frac{\partial v_{2h}}{\partial x}) \\
\quad -c(E^{I,n}_{\textbf{u}},\textbf{u}^{n,\theta}_h,\textbf{v}_h)-c(E^{A,n}_{\textbf{u}},\textbf{u}^{n,\theta}_h,\textbf{v}_h)- c(\textbf{u}^{n},E^{I,n,\theta}_{\textbf{u}},\textbf{v}_h) \\
\quad - c(\textbf{u}^{n},E^{A,n, \theta}_{\textbf{u}},\textbf{v}_h)-a_{LT}(E^{I,n,\theta}_c,d_h)-a_{NLT}(E^{I,n}_\textbf{u},c_h^{n,\theta},d_h)-\\
\quad a_{NLT}(E^{A,n}_\textbf{u},c_h^{n,\theta},d_h)-a_{NLT}(\textbf{u}^{n}, E^{I,n,\theta}_c,d_h)-a_{NLT}(\textbf{u}^{n}, E^{A,n,\theta}_c,d_h)-\\
\quad a_{PL}(\eta^n; E^{I,n,\theta}_{\textbf{u}}, \textbf{v}_h)-a_{PL}(\eta^n-\eta_h^n ;\textbf{u}_h^{n,\theta},\textbf{v}_h) -I_1-I_2-I_3-I_4+ (\textbf{TE}^{n,\theta}, \textbf{V}_h)
\end{multline}
where $\eta^n$ and $\eta_h^n$ denote $\eta(c^n,\textbf{u}^n)$ and $\eta(c_h^n,\textbf{u}^n_h)$ respectively and the terms $I_i$, for $i=1,...,4$ will be expressed and estimated together later. The above rearrangement plays an important role in this estimation since the combined result of lower bound and upper bound, to be found now,  for the terms in left hand side and right hand side of (33) respectively will give the desired estimation. Since the above equation (33) holds for all $\textbf{V}_h \in (V_h)^m \times  Q_h \times V_h$, therefore in each term we replace $\textbf{v}_{h}, q_h,d_h$ by $E^{A,n,\theta}_{\textbf{u}},E^{A,n,\theta}_{p},E^{A,n,\theta}_{c}$ respectively. From now onward we estimate each expression after considering the replacements directly. \vspace{1mm}\\
Applying (29) on the first two terms of $LHS$ in (33) and applying the assumption \textbf{(i)} on the diffusion coefficients we can easily find a lower bound for the first three terms in $LHS$. Now estimations of the remaining terms involve the exercise of the assumption $\textbf{(iv)}$ on $\eta(\cdot,\cdot)$. Therefore the estimation of the fourth term follows this way:
\begin{equation}
\begin{split}
& \int_{\Omega} \eta^n \{ 2(\frac{\partial E^{A,n,\theta}_{u1}}{\partial x})^2 + (\frac{\partial E^{A,n,\theta}_{u1}}{\partial y})^2+ (\frac{\partial E^{A,n,\theta}_{u2}}{\partial x})^2+ 2(\frac{\partial E^{A,n,\theta}_{u2}}{\partial y})^2\}\\
& = \int_{\Omega} K e^{Bc^n} \{ 2 (\frac{\partial u_1^n}{\partial x})^2 + 2 (\frac{\partial u_2^n}{\partial y})^2 + (\frac{\partial u_1^n}{\partial y}+ \frac{\partial u_2^n}{\partial x})^2 \}^{\frac{m-1}{2}} \{ 2(\frac{\partial E^{A,n,\theta}_{u1}}{\partial x})^2 +\\
& \quad (\frac{\partial E^{A,n,\theta}_{u1}}{\partial y})^2+ (\frac{\partial E^{A,n,\theta}_{u2}}{\partial x})^2+ 2(\frac{\partial E^{A,n,\theta}_{u2}}{\partial y})^2\} \geq \eta_l  \mid E^{A,n,\theta}_{\textbf{u}} \mid_1^2
\end{split}
\end{equation}
where applying the assumption $\textbf{(iv)}$ on the continuous solutions and it's derivatives 
\begin{equation}
    \eta_l=
    \begin{cases}
      K e^{B c_l} \{ 2 u_{1l}^2+ 2u_{2l}'^2 + (u_{1l}'+u_{2l})^2 \}^{\frac{m-1}{2}}, & \text{if}\ m \geq 1 \\
      K e^{B c_l} \{ 2 u_{1s}^2+ 2u_{2s}^2 + (u_{1s}'+u_{2s})^2 \}^{\frac{m-1}{2}}, & \text{if}\ m<1
    \end{cases}
  \end{equation}
where $c_l, u_{1l}, u_{1l}', u_{2l}, u_{2l}'$ are $minimum$ of $c^n, \frac{\partial u_1^n}{\partial x}, \frac{\partial u_1^n}{\partial y}, \frac{\partial u_2^n}{\partial x}, \frac{\partial u_2^n}{\partial y}$ respectively and $ u_{1s}, u_{1s}', u_{2s}, u_{2s}'$ are $maximum$ of $ \frac{\partial u_1^n}{\partial x}, \frac{\partial u_1^n}{\partial y}, \frac{\partial u_2^n}{\partial x}, \frac{\partial u_2^n}{\partial y}$ respectively for each n taken over $\Omega$. 
Similar to this derivation the estimated result of the next term is
\begin{equation}
 \int_{\Omega} \eta(c^{n}, \textbf{u}^{n})\{(E^{A,n,\theta}_{u1})^2+ (E^{A,n,\theta}_{u2})^2\} \geq \eta_l  \|E^{A,n,\theta}_{\textbf{u}} \|^2
\end{equation}
where $\eta_l$ carries the same values as earlier for each of the cases considered above.
Combining all the above results in (33) we have
\begin{equation}
\begin{split}
LHS & \geq \frac{\rho}{2 dt}  (\|E^{A,n+1}_{\textbf{u}}\|^2- \|E^{A,n}_{\textbf{u}}\|^2)+ \frac{1}{2 dt}(\|E^{A,n+1}_c\|^2- \|E^{A,n}_c\|^2)+ \eta_l  \| E^{A,n,\theta}_{\textbf{u}}\|_1^2 \\
& \quad +D_l \mid E^{A,n,\theta}_c \mid_1^2 + \alpha \|E^{A,n,\theta}_c\|^2
\end{split}
\end{equation}
where $D_l$= $ \underset{i=1,...,m}{min} \underset{\Omega}{inf} D_i $. \\
Now let us estimate the terms in the $RHS$ in (33).  After the substitution of $\textbf{V}_h$ by the respective auxiliary error parts it can be seen that first two terms in the $RHS$ get canceled out with each other and estimations of the bilinear terms are quite straight-forward with the help of the Cauchy-Schwarz inequality, Young's Inequality, standard interpolation estimate (27) and assumptions \textbf{(i)} on the diffusion coefficients wherever required. Hence we skip the detailed derivations of those terms and mention the final results only.
\begin{equation}
\begin{split}
b(E^{A,n,\theta}_{\textbf{u}},E^{I,n,\theta}_p) &  \leq  \frac{1}{2 \epsilon_1} \mid E^{A,n,\theta}_{\textbf{u}} \mid_1^2+ \epsilon_1 C^2 h^2(\frac{1+\theta}{2}\| p^{n+1}\|_1 + \frac{1-\theta}{2}\| p^n \|_1)^2 \\
-a_{LT}(E^{I,n,\theta}_c, E^{A,n,\theta}_c) 
& \leq \frac{D_m}{2 \epsilon_2}  \mid E^{A,n,\theta}_{c} \mid_1^2+ \frac{D_m \epsilon_2}{2} C^2 h^2 (\frac{1+\theta}{2}\| c^{n+1} \|_2+\frac{1-\theta}{2}\| c^n \|_2)^2
\end{split}
\end{equation}
Now we estimate the trilinear terms $c(\cdot,\cdot,\cdot)$ using it's properties \textbf{(a)} and \textbf{(b)} associated with it. We again apply few of the above mentioned inequalities to derive the estimates.
\begin{equation}
\begin{split}
-c(E^{I,n}_{\textbf{u}},\textbf{u}^{n,\theta}_h, E^{A,n,\theta}_{\textbf{u}}) & = c(E^{I,n}_{\textbf{u}},E^{I,n,\theta}_{\textbf{u}}, E^{A,n,\theta}_{\textbf{u}})+ c(E^{I,n}_{\textbf{u}},E^{A,n,\theta}_{\textbf{u}}, E^{A,n,\theta}_{\textbf{u}}) -c(E^{I,n}_{\textbf{u}},\textbf{u}^{n,\theta}, E^{A,n,\theta}_{\textbf{u}}) \\
& \leq C \|E^{I,n}_{\textbf{u}}\| \|E^{I,n,\theta}_{\textbf{u}}\|_1 \|E^{A,n,\theta}_{\textbf{u}}\|_1 +C \|E^{I,n,\theta}_{\textbf{u}}\| \|\textbf{u}^{n,\theta}\|_2 \|E^{A,n,\theta}_{\textbf{u}}\|_1 \\
& \leq h^4 \frac{C_2 \epsilon_3 }{2}( \|\textbf{u}^{n}\|+1) (\frac{1+\theta}{2} \|\textbf{u}^{n+1}\|_2 + \frac{1-\theta}{2} \|\textbf{u}^{n}\|_2)^2 +\frac{C_2}{ \epsilon_3} \|E^{A,n,\theta}_{\textbf{u}}\|_1^2 \\
\end{split}
\end{equation}
and applying the same properties and inequalities we have the following estimated results for the similar terms,
\begin{equation}
\begin{split}
-c(\textbf{u}^{n},E^{I,n,\theta}_{\textbf{u}}, E^{A,n,\theta}_{\textbf{u}}) & \leq  \frac{ \epsilon_3}{2} C_2 h^2 (\frac{1+\theta}{2} \|\textbf{u}^{n+1}\|_2 + \frac{1-\theta}{2} \|\textbf{u}^{n}\|_2)^2 + \frac{C_2}{2 \epsilon_3} \|E^{A,n,\theta}_{\textbf{u}}\|^2\\
-c(E^{A,n}_{\textbf{u}},\textbf{u}^{n,\theta}_h, E^{A,n,\theta}_{\textbf{u}}) & \leq C_2' \|E^{A,n,\theta}_{\textbf{u}}\|_1^2 \\
-c(\textbf{u}^{n},E^{A,n,\theta}_{\textbf{u}}, E^{A,n,\theta}_{\textbf{u}}) & =0 
\end{split}
\end{equation}
The estimation of the another type of trilinear term $a_{NLT}(\cdot,\cdot,\cdot)$ involves the use of assumption $\textbf{(iii)}$ on the continuous solutions $\textbf{u}^n$ and $c^n$ and the Cauchy-Schwarz and Youngs inequalities as follows:
\begin{equation}
\begin{split}
-a_{NLT}(\textbf{u}^{n}, E^{I,n,\theta}_c, E^{A,n,\theta}_c) & \leq \frac{\bar{C_1} + \bar{C_2}}{2} \{ \frac{C^2 h^2}{\epsilon_4} (\frac{1+\theta}{2} \|c^{n+1}\|_2 + \frac{1-\theta}{2} \|c^{n}\|_2)^2 + \epsilon_4 \|E^{A,n,\theta}_c\|^2 \}\\
-a_{NLT}(\textbf{u}^{n}, E^{A,n,\theta}_c, E^{A,n,\theta}_c) 
& \leq \frac{\bar{C_1}+\bar{C_2}}{2 } \mid E^{A,n,\theta}_c \mid_1^2 + \frac{\bar{C_1}+ \bar{C_2}}{2}  \|E^{A,n,\theta}_c\|^2\\
\end{split}
\end{equation}
Applying assumption \textbf{(iv)} we have $\bar{C_1}$= $\underset{\Omega}{sup}$ $\mid u_1 \mid$ and $\bar{C_2}$= $\underset{\Omega}{sup}$ $\mid u_2 \mid$ for each $n$. 
Estimation of the next term is slightly different from the previous one by splitting it into three parts as follows:
\begin{equation}
\begin{split}
-a_{NLT}(E_{\textbf{u}}^{I,n}, c_h^{n,\theta}, E^{A,n,\theta}_c) & \leq \mid a_{NLT}(E_{\textbf{u}}^{I,n}, E_c^{I,n,\theta}, E^{A,n,\theta}_c) \mid + \mid a_{NLT}(E_{\textbf{u}}^{I,n}, E_c^{A,n,\theta}, E^{A,n,\theta}_c) \mid \\
& \quad + \mid a_{NLT}(E_{\textbf{u}}^{I,n}, c^{n,\theta}, E^{A,n,\theta}_c) \mid
\end{split}
\end{equation}
Using the Sobolev, the Cauchy-Schwarz and the Young inequalities subsequently the estimations of these three parts are:
\begin{equation}
\begin{split}
\mid a_{NLT}(E_{\textbf{u}}^{I,n}, E_c^{I,n,\theta}, E^{A,n,\theta}_c) \mid & \leq  \epsilon_4 \|E^{A,n,\theta}_c\|^2 + \frac{C^4 h^4}{2 \epsilon_4}  \| \textbf{u}^n\|_2^2 (\frac{1+\theta}{2} \|c^{n+1}\|_2 + \frac{1+\theta}{2} \|c^{n}\|_2)^2 \\
\mid a_{NLT}(E^{I,n}_\textbf{u}, E^{A,n, \theta}_c, E^{A,n,\theta}_c) \mid & \leq \epsilon_4 \|E^{A,n,\theta}_c\|^2 + \frac{C_1'}{2 \epsilon_4} \|\frac{\partial E^{A,n,\theta}_c}{\partial x}\|^2+  \frac{C_2'}{2 \epsilon_5} \|\frac{\partial E^{A,n,\theta}_c}{\partial y}\|^2\\
\mid a_{NLT}(E^{I,n}_\textbf{u}, c^{n, \theta}, E^{A,n,\theta}_c) \mid
& \leq \epsilon_4 \|E^{A,n,\theta}_c\|^2 + \frac{C^2 h^2}{2 \epsilon_4} \|c^{n,\theta}\|_2^2 \{  (\frac{1+\theta}{2} \|\textbf{u}^{n+1}\|_2+  \frac{1-\theta}{2} \|\textbf{u}^{n}\|_2)^2 \}
\end{split}
\end{equation}
On expanding the estimation of the term $a_{PL}(\cdot; \cdot,\cdot)$ can be easily carried out with the help of above mentioned standard inequalities and the estimated results are as follows:
\begin{equation}
\begin{split}
-a_{PL}(\eta^n; E^{I,n,\theta}_\textbf{u}, E^{A,n,\theta}_\textbf{u}) & \leq \frac{\eta_s}{\epsilon_5} C^2 h^2 \{ (\frac{1+\theta}{2} \|\textbf{u}^{n+1}\|_2 + \frac{1-\theta}{2} \|\textbf{u}^{n}\|_2)^2 \} + \eta_s \epsilon_5 \mid E^{A,n,\theta}_{\textbf{u}}\mid_1^2\\
\end{split}
\end{equation}
where applying the assumption $\textbf{(iv)}$ on the continuous solutions and it's derivatives 
\begin{equation}
    \eta_s=
    \begin{cases}
       K e^{B c_s} \{ 2 u_{1s}^2+ 2 u_{2s}'^2+ (u_{1s}'+ u_{2s})^2 \}^{\frac{m-1}{2}}, & \text{if}\ m \geq 1 \\
      K e^{B c_s} \{ 2 u_{1l}^2+ 2 u_{2l}'^2+ (u_{1l}'+ u_{2l})^2 \}^{\frac{m-1}{2}}, & \text{if}\ m<1
    \end{cases}
  \end{equation}
Now in a likewise manner the estimation of the next terms:
\begin{equation}
\begin{split}
& -a_{PL}(\eta^{n} - \eta^{n}_h; \textbf{u}_h^{n,\theta}, E^{A,n,\theta}_\textbf{u})  \\
& \leq \mid a_{PL}(\eta^{n}; \textbf{u}^{n,\theta}, E^{A,n,\theta}_\textbf{u}) \mid + \mid a_{PL}(\eta^{n}; E_\textbf{u}^{I,n,\theta}, E^{A,n,\theta}_\textbf{u}) \mid  + \mid a_{PL}(\eta^{n}; E_\textbf{u}^{A,n,\theta}, E^{A,n,\theta}_\textbf{u}) \mid  \\
& \leq \frac{\eta_s}{\epsilon_5} C^2 h^2  (\frac{1+\theta}{2} \|\textbf{u}^{n+1}\|_2 + \frac{1-\theta}{2} \|\textbf{u}^{n}\|_2)^2  + \eta_s \epsilon_5  \mid E^{A,n,\theta}_{\textbf{u}} \mid_1^2 
\end{split}
\end{equation}
The estimations of the next terms follow the applications of assumption $\textbf{(iv)}$ on the viscosity expression and $Poincare$ inequality subsequently as follows:
\begin{equation}
\begin{split}
\int_{\Omega} \eta^n \{ (E^{A,n,\theta}_{u1})^2 + (E^{A,n,\theta}_{u2})^2 \} & \leq  \eta_s C_P (\mid E^{A,n,\theta}_{u1}\mid_1^2+\mid E^{A,n,\theta}_{u2}\mid_1^2) \\
-\int_{\Omega} 2 \eta^{n} \frac{\partial E^{A,n,\theta}_{u2}}{\partial x} \frac{\partial E^{A,n,\theta}_{u1}}{\partial y} & \leq \frac{\eta_s}{2 \epsilon_6} \| \frac{\partial E^{A,n,\theta}_{u2}}{\partial x}\|^2 + \frac{\eta_s \epsilon_6}{2 } \| \frac{\partial E^{A,n,\theta}_{u1}}{\partial y}\|^2 
\end{split}
\end{equation}
where $\eta_s$ satisfies (45). Now our job is to estimate those terms defined on the sub domain $\Omega_k$. Here we use that observation made during the proof of stability result: by the virtue of the choice of the finite element spaces $V_h$ and $Q_h$, we can clearly say that over each element sub domain every function belonging to that spaces and their first and second order derivatives all are bounded functions.
\begin{equation}
\begin{split}
I_1 & = \sum_{k=1}^{n_{el}}[(\tau_k' (M \partial_t (\textbf{U}^n-\textbf{U}^n_h)+ \mathcal{L}(\textbf{u}^n, \eta^n ; \textbf{U}^{n,\theta})- \mathcal{L}(\textbf{u}_h^n, \eta_h^n; \textbf{U}_h^{n,\theta})), -\mathcal{L}^*(\textbf{u}_h, \eta_h; E^{A,n,\theta}_{\textbf{U}}))_{\Omega_k} \\
& = \sum_{k=1}^{n_{el}} (\tau_{1k}' ( \partial_t (\textbf{u}^n-\textbf{u}^n_h)+ \rho(\textbf{u}^n \cdot \nabla) \textbf{u}^{n,\theta} + \nabla p^{n,\theta}- \nabla \cdot 2 \eta \textbf{D}(\textbf{u}^{n,\theta})-\rho(\textbf{u}^n_h \cdot \nabla) \textbf{u}^{n,\theta}_h- \\
& \quad \nabla p^{n,\theta}_h-  \nabla \cdot 2 \eta \textbf{D}(\textbf{u}^{n,\theta}_h), \rho(\textbf{u}^n \cdot \nabla) E^{A,n,\theta}_{\textbf{u}}+ \nabla E^{A,n,\theta}_p- \nabla \cdot 2 \eta \textbf{D}(E^{A,n,\theta}_{\textbf{u}}))_{\Omega_k} +\\
& \quad \tau_{2k} (\nabla \cdot (\textbf{u}^{n,\theta}-\textbf{u}_h^{n,\theta}), \nabla \cdot E^{A,n,\theta}_{\textbf{u}})_{\Omega_k} + \tau_{3k}' (\partial_t (c^{n,\theta}-c_h^{n,\theta})- \nabla \cdot \tilde{\nabla} (c^{n,\theta}-c^{n,\theta}_h) + \\
& \quad \textbf{u}^n \cdot \nabla c^{n,\theta}- \textbf{u}^n_h \cdot \nabla c_h^{n,\theta} + \alpha (c^{n,\theta}-c^{n,\theta}_h), \nabla \cdot \tilde{\nabla} E^{A,n,\theta}_c + \textbf{u}^n \cdot \nabla E^{A,n,\theta}_c - \alpha E^{A,n,\theta}_c)_{\Omega_k}]
\end{split}
\end{equation}
Subsequently applying error splitting, the above observation on the terms belonging to the finite element spaces, the standard interpolation estimate (27) and the consequence of assumption $\textbf{(iii)}$ on the continuous solutions we have reached at the following result.
\begin{equation}
\begin{split}
I_1 &  \leq \mid \tau_1 \mid \bar{C}_1^1(h,\textbf{u}^{n,\theta},p^{n,\theta}) + h^2 \mid \tau_2 \mid  \bar{C}^1_2(\textbf{u}^{n,\theta}) + \mid \tau_3 \mid \bar{C}_3^1(h,c^{n,\theta})
\end{split}
\end{equation}
where the parameters $\bar{C}_i^1$ (for $i=1,2,3$) are resultant summations of positive constants, obtained due to applying the above mentioned results, multiplied with different norms of exact solutions. Now before estimating the remaining terms we first expand the other notations $I_i$ for i=2,3,4 and later using the similar argument exercised to estimate $I_1$ we mention the estimated results of them.
\begin{equation}
\begin{split}
I_2 & = \sum_{k=1}^{n_{el}}((I- \tau_k^{-1} \tau_k') (M \partial_t (\textbf{U}^n-\textbf{U}^n_h)+ \mathcal{L}(\textbf{u}^n, \eta^n; \textbf{U}^{n,\theta})- \mathcal{L}(\textbf{u}_h^n, \eta_h^n; \textbf{U}_h^{n,\theta})),-E^{A,n,\theta}_{\textbf{U}})_{\Omega_k} \\
\end{split}
\end{equation}
On expansion of the next term
\begin{equation}
\begin{split}
I_3 & =\sum_{k=1}^{n_{el}} (\tau_k^{-1}\tau_k' \textbf{d}, E^{A,n,\theta}_{\textbf{U}})_{\Omega_k} \\
&= \sum_{k=1}^{n_{el}} (\tau_k^{-1}\tau_k' \sum_{i=1}^{n+1} ((\frac{1}{dt}M\tau_k')^i)(M\partial_t (\textbf{U}^n-\textbf{U}^n_h) +  \mathcal{L}(\textbf{u}^n, \eta^n; \textbf{U}^{n,\theta})- \mathcal{L}(\textbf{u}_h^n, \eta_h^n; \textbf{U}_h^{n,\theta})), E^{A,n,\theta}_{\textbf{U}})_{\Omega_k} \\
& \leq \sum_{k=1}^{n_{el}} (\tau_k^{-1}\tau_k' \sum_{i=1}^{\infty} ((\frac{1}{dt}M\tau_k')^i)(M\partial_t (\textbf{U}^n-\textbf{U}^n_h) +  \mathcal{L}(\textbf{u}^n, \eta^n; \textbf{U}^{n,\theta})- \mathcal{L}(\textbf{u}_h^n, \eta_h^n; \textbf{U}_h^{n,\theta})), E^{A,n,\theta}_{\textbf{U}})_{\Omega_k} \\
& = \sum_{k=1}^{n_{el}}  [ \frac{\rho \tau_{1k}}{(dt+\rho \tau_{1k})}  ( \partial_t (\textbf{u}^n-\textbf{u}^n_h)+ \rho(\textbf{u}^n \cdot \nabla) \textbf{u}^{n,\theta} + \nabla p^{n,\theta}- \nabla \cdot 2 \eta \textbf{D}(\textbf{u}^{n,\theta})-\rho(\textbf{u}^n_h \cdot \nabla) \textbf{u}^{n,\theta}_h- \\
& \quad \nabla p^{n,\theta}_h-  \nabla \cdot 2 \eta \textbf{D}(\textbf{u}^{n,\theta}_h), E^{A,n,\theta}_{\textbf{u}})_{\Omega_k}+ \frac{\tau_{3k}}{(dt+\tau_{3k})}(\partial_t (c^{n,\theta}-c_h^{n,\theta})- \nabla \cdot \tilde{\nabla} (c^{n,\theta}-c^{n,\theta}_h) +  \textbf{u}^n \cdot \nabla c^{n,\theta}\\
& \quad -\textbf{u}^n_h \cdot \nabla c_h^{n,\theta} + \alpha (c^{n,\theta}-c^{n,\theta}_h), E^{A,n,\theta}_c)_{\Omega_k} ]
\end{split}
\end{equation}
For $dt>0$, $\frac{\rho \tau_1}{dt+\rho \tau_1} < 1$ and $\frac{\tau_3}{dt+ \tau_3} < 1$, which implies $\frac{\rho \tau_1'}{dt} < 1$ and $\frac{\tau_3'}{dt} < 1$ and therefore the series $\sum_{i=1}^{\infty}(\frac{\rho}{dt} \tau_1' )^i$ and $\sum_{i=1}^{\infty}(\frac{1}{dt} \tau_3' )^i$ converges to $\frac{\rho \tau_1' }{(dt- \tau_1' )}=\frac{\rho \tau_{1k}}{dt}$ and $\frac{\tau_3' }{(dt- \tau_3' )}=\frac{\tau_{3k}}{dt}$ respectively.
\begin{equation}
\begin{split}
I_4 & = \sum_{k=1}^{n_{el}}(\tau_k' \textbf{d},-\mathcal{L}^*(\textbf{u}_h;E^{A,n,\theta} _{\textbf{U}}))_{\Omega_k} \\
& \leq   \sum_{k=1}^{n_{el}}[ (\frac{\rho \tau_{1k}^2}{(dt+\rho \tau_{1k})} ( \partial_t (\textbf{u}^n-\textbf{u}^n_h)+ \rho(\textbf{u}^n \cdot \nabla) \textbf{u}^{n,\theta} + \nabla p^{n,\theta}- \nabla \cdot 2 \eta \textbf{D}(\textbf{u}^{n,\theta})-\rho(\textbf{u}^n_h \cdot \nabla) \textbf{u}^{n,\theta}_h- \\
& \quad \nabla p^{n,\theta}_h-  \nabla \cdot 2 \eta \textbf{D}(\textbf{u}^{n,\theta}_h), \rho(\textbf{u}^n \cdot \nabla) E^{A,n,\theta}_{\textbf{u}}+ \nabla E^{A,n,\theta}_p- \nabla \cdot 2 \eta \textbf{D}(E^{A,n,\theta}_{\textbf{u}}))_{\Omega_k} +  \\
& \quad \frac{\tau_{3k}^2}{(dt+\tau_{3k})}(\partial_t (c^{n,\theta}-c_h^{n,\theta})- \nabla \cdot \tilde{\nabla} (c^{n,\theta}-c^{n,\theta}_h) + \textbf{u}^n \cdot \nabla c^{n,\theta}- \textbf{u}^n_h \cdot \nabla c_h^{n,\theta} + \alpha (c^{n,\theta}-c^{n,\theta}_h),\\
& \quad   \nabla \cdot \tilde{\nabla} E^{A,n,\theta}_c + \textbf{u}^n \cdot \nabla E^{A,n,\theta}_c - \alpha E^{A,n,\theta}_c)_{\Omega_k}]
\end{split}
\end{equation}
Clearly on expansion $I_3$ and $I_4$ look same as $I_2$ and $I_1$ respectively and hence the estimated results of the remaining terms are as follows:
\begin{equation}
\begin{split}
I_2 & \leq \mid \tau_1 \mid \bar{C}_1^2(h,\textbf{u}^{n,\theta},p^{n,\theta}) + \mid \tau_3 \mid \bar{C}_3^2(h,c^{n,\theta}) \\
I_3 & \leq \mid \tau_1 \mid \bar{C}_1^3(h,\textbf{u}^{n,\theta},p^{n,\theta})  + \mid \tau_3 \mid \bar{C}_3^3(h,c^{n,\theta})\\
I_4 & \leq \mid \tau_1 \mid \bar{C}_1^4(h,\textbf{u}^{n,\theta},p^{n,\theta})  + \mid \tau_3 \mid \bar{C}_3^4(h,c^{n,\theta})
\end{split}
\end{equation}
where the parameters $\bar{C}_1^i, \bar{C}_3^i$ for $i=2,3,4$ are obtained similarly. Again applying the Cauchy-Schwarz and Young's inequality once again to estimate the truncation error terms as follows:
\begin{equation}
\begin{split}
(\textbf{TE}^{n,\theta}, E^{A,n,\theta}_{\textbf{U}})
& \leq \frac{ \epsilon_7}{2} \|\textbf{TE}^{n,\theta}\|^2 + \frac{1}{2 \epsilon_7} (\|E^{A,n,\theta}_{\textbf{u}}\|^2 +\|E^{A,n,\theta}_{c}\|^2)
\end{split}
\end{equation}
Finally  we have completed finding bounds for each of the terms in the right hand side of (33). Now we combine all the results obtained in (37)-(54) into the equation (33)
and take out the suitable required terms to the left hand side. Next to that  multiplying both sides by 2$dt$ and taking summation over the time steps for n=0,1,...,$(N-1)$ to both the sides successively finally we have (33) as follows:
\begin{multline}
\rho \sum_{n=0}^{N-1} (\|E^{A,n+1}_{\textbf{u}}\|^2- \|E^{A,n}_{\textbf{u}}\|^2)+   \sum_{n=0}^{N-1} (\|E^{A,n+1}_c\|^2- \|E^{A,n}_c\|^2)+ (2 \eta_l-\frac{1}{\epsilon_1}-\frac{3C_2}{\epsilon_3}-2C_2'-\\
\quad 4 \eta_s \epsilon_5-2 \eta_s C_P-\frac{1}{\epsilon_7})  \sum_{n=0}^{N-1} \mid E^{A,n,\theta}_{\textbf{u}}\mid_1^2 dt+ (2 \eta_l-\frac{3C_2}{\epsilon_3}-2C_2'-\frac{1}{\epsilon_7})  \sum_{n=0}^{N-1} \| E^{A,n,\theta}_{\textbf{u}}\|_1^2 dt+  \\
\quad \{D_l- \frac{D_m}{\epsilon_2}- (\bar{C}_1+\bar{C}_2) \epsilon_4-\frac{1}{\epsilon_7} \}  \sum_{n=0}^{N-1} \mid E^{A,n,\theta}_{c}\mid_1^2 dt + (2 \alpha- 4 \epsilon_4-\frac{1}{\epsilon_7})  \sum_{n=0}^{N-1} \| E^{A,n,\theta}_{c}\|_1^2 dt\\
\leq h^2  \sum_{n=0}^{N-1} [2\epsilon_1 C^2 (\frac{1+\theta}{2} \|p^{n+1}\|_1+\frac{1-\theta}{2} \|p^{n}\|_1)^2+C^2( D_m \epsilon_2+ \frac{\bar{C}_1+\bar{C}_2}{\epsilon_4}+ \frac{C^2 h^2}{\epsilon_4} \|\textbf{u}^n\|_2^2) \\
\quad (\frac{1+\theta}{2} \|c^{n+1}\|_2+\frac{1-\theta}{2} \|c^{n}\|_2)^2 + (h^2C_2 \epsilon_3+C_2 \epsilon_3 + \frac{C^2}{\epsilon_4} \|c^{n,\theta}\|_2^2+\frac{2\eta_s C^2}{\epsilon_5}) \\
\quad (\frac{1+\theta}{2} \|\textbf{u}^{n+1}\|_2+\frac{1-\theta}{2} \|\textbf{u}^{n}\|_2)^2 ] dt +  \mid \tau_1 \mid \sum_{n=0}^{N-1}  \sum_{i=1}^3
 \bar{C}_1^i(h,\textbf{u}^{n,\theta},p^{n,\theta}) dt+\\
  \mid \tau_2 \mid h^2  \sum_{n=0}^{N-1} \bar{C}^1_2(\textbf{u}^{n,\theta}) dt+ \mid \tau_3 \mid \sum_{n=0}^{N-1} \sum_{i=1}^3  \bar{C}_1^i(h,\textbf{u}^{n,\theta},p^{n,\theta}) dt +\epsilon_7 \sum_{n=0}^{N-1} \|\textbf{TE}^{n,\theta}\|^2 dt
\end{multline}
We can choose the values of the arbitrary parameters $\epsilon_i$s for $i=1,2,...,9$ in such a manner that we can make all the coefficients in the left hand side positive. Now after taking minimum of all the coefficients in left hand side, let us divide both the sides with that minimum, which turns out to be a positive real number. Applying (3.7) on the truncation error terms we will finally arrive at the following expression using the fact that $\tau_1$ and $\tau_3$ are of order $h^2$:
\begin{equation}
\boxed{ \|E^{A}_{\textbf{u}}\|_{\textbf{M}}^2+ \|E^{A}_{c}\|_{\textbf{M}}^2 \leq C(T,\textbf{u},p,c) (h^2+dt^{2r})}
\end{equation}
where
\begin{equation}
    r=
    \begin{cases}
      1, & \text{if}\ \theta=1 \\
      2, & \text{if}\ \theta=0
    \end{cases}
  \end{equation}
We have used the fact that $\sum_{n=0}^{N-1} \int_{t^n}^{t^{n+1}} M dt \leq M T  $. This completes the first part of the proof. \vspace{2mm}\\
\textbf{Second part} Using this above result we are going to estimate auxiliary error part of pressure. We will use inf-sup condition to find estimate for $E_p^A$. Applying Galerkin orthogonality on flow problem and later splitting the error we have obtained
\begin{multline}
 b(\textbf{v}_h,p-I^h_p p)+ b(\textbf{v}_h,I^h_p p- p_h)=(\partial_t E^A_{u1}, v_{1h})+ (\partial_t E^A_{u2}, v_{2h}) + c(E^I_{\textbf{u}},\textbf{u},\textbf{v}_h)\\
+c(E^A_{\textbf{u}},\textbf{u},\textbf{v}_h)+ c(\textbf{u}_h,E^I_{\textbf{u}},\textbf{v}_h)+ c(\textbf{u}_h,E^A_{\textbf{u}},\textbf{v}_h)  \\
+ a_{PL}(\eta(c,\textbf{u});\textbf{u},\textbf{v}_h)-a_{PL}(\eta(c_h,\textbf{u}_h);\textbf{u}_h,\textbf{v}_h)
\end{multline}
Applying the inclusion $\bigtriangledown \cdot V_h \subset Q_h$ and the property of the $L^2$ orthogonal projection of $I^h_p$ we have 
\begin{equation}
b(\textbf{v}_h, p-I^h_p p)= \int_{\Omega}(p-I^h_p p)(\bigtriangledown \cdot \textbf{v}_h)=0
\end{equation}
Now according to inf-sup condition we will have the following expression
\begin{equation}
\begin{split}
\|I^h_p p-p_h\|_{\textbf{N}}^2 = \|E_p^A\|_{\textbf{N}}^2 
& =  \sum_{n=0}^{N-1}  \|E_p^{A,n,\theta}\|^2 dt \\
& \leq  \sum_{n=0}^{N-1} \underset{\textbf{v}_h}{sup} \frac{b(\textbf{v}_h, E_p^{A,n,\theta})}{\|\textbf{v}_h\|_1} dt
\end{split}
\end{equation}
Applying (59) on (58) we have
\begin{equation}
\begin{split}
\sum_{n=0}^{N-1} b(\textbf{v}_h, E_p^{A,n,\theta}) dt & = \sum_{n=0}^{N-1} \{ (\frac{E^{A,n+1}_{\textbf{u}}-E^{A,n}_{\textbf{u}}}{dt},\textbf{v}_{h})+ c(E^{I,n}_{\textbf{u}},\textbf{u}^{n,\theta},\textbf{v}_h) +\\
& \quad  c(E^{A,n}_{\textbf{u}},\textbf{u}^{n,\theta},\textbf{v}_h)+  c(\textbf{u}_h^{n}, E^{I,n,\theta}_\textbf{u},\textbf{v}_h)+ c(\textbf{u}_h^{n}, E^{A,n,\theta}_\textbf{u},\textbf{v}_h)\\
& \quad +a_{PL}(\eta;\textbf{u},\textbf{v}_h)-  a_{PL}(\eta_h;\textbf{u}_h,\textbf{v}_h)\}dt
\end{split}
\end{equation}
Now applying the results obtained in the previous part and substituting that in (60) we have the estimate for the pressure term
\begin{equation}
\|I_hp-p_h\|_{\textbf{N}}^2 \leq C(T,\textbf{u},p,c) (h^2+dt^{2r})
\end{equation}
Now combining the results obtained in the first and second part we have finally arrived at the auxiliary error estimate as follows
\begin{equation}
\|E^A_{\textbf{u}}\|^2_{\textbf{M}} + \|E^A_p\|_{\textbf{N}}^2  + \|E^A_{c}\|^2_{\textbf{M}} \leq C(T,\textbf{u},p,c) (h^2+ dt^{2r})
\end{equation}
where
\begin{equation}
    r=
    \begin{cases}
      1, & \text{if}\ \theta=1 \\
      2, & \text{if}\ \theta=0
    \end{cases}
  \end{equation}
This completes the proof.
\end{proof}
\begin{theorem}(Apriori error estimate)
Assuming the same condition as in the previous theorem, 
\begin{equation}
\|\textbf{u}-\textbf{u}_h\|_{\textbf{M}}^2+\|p-p_h\|_{\textbf{N}}^2 + \|c-c_h\|^2_{\textbf{M}} \leq C' (h^2+ dt^{2r})
\end{equation}
where $C'$ depends on T, $\textbf{u}$,p,c and
\begin{equation}
    r=
    \begin{cases}
      1, & \text{if}\ \theta=1 \\
      2, & \text{if}\ \theta=0
    \end{cases}
  \end{equation}
\end{theorem}
\begin{proof}
By applying triangle inequality, the interpolation inequalities and the result of the previous theorem we will have,
\begin{equation}
\begin{split}
& \|\textbf{u}-\textbf{u}_h\|_{\textbf{M}}^2+\|p-p_h\|_{\textbf{N}}^2 + \|c-c_h\|^2_{\textbf{M}} \\
& \leq \bar{C} (\|E^I_{\textbf{u}}\|_{\textbf{M}}^2 +\|E^I_{p}\|_{\textbf{N}}^2 +  \|E^I_{c}\|_{\textbf{M}}^2+ 
 \|E^A_{\textbf{u}}\|_{\textbf{M}}^2 +\|E^A_{p}\|_{\textbf{N}}^2 +  \|E^A_{c}\|_{\textbf{M}}^2\\
& \leq C'(T,\textbf{u},p,c)(h^2+ dt^{2r})
\end{split}
\end{equation}
This completes apriori error estimation.
\end{proof} 

\subsection{Aposteriori error estimation}
\begin{theorem} 
 Assume the time step $dt$ is sufficiently small and positive, and the coefficients in (1)-(4) satisfy the assumptions \textbf{(i)}-\textbf{(ii)}. Then for sufficiently regular continuous solutions $(\textbf{u},p,c)$ satisfying the assumptions \textbf{(iii)}-\textbf{(iv)} and computed solutions $(\textbf{u}_h,p_h,c_h)$ belonging to $V_h \times V_h \times Q_h \times V_h$ satisfying (19), there exists a constant $\bar{C}$, depending upon the computed solutions $\textbf{u}_h,p_h,c_h$ such that
\begin{equation}
 \|\textbf{u}-\textbf{u}_h\|_{\textbf{M}}^2+\|p-p_h\|_{\textbf{N}}^2 + \|c-c_h\|^2_{\textbf{M}} \leq \bar{C}(\textbf{R}) (h^2+ dt^{2r})
\end{equation}
where
\begin{equation}
    r=
    \begin{cases}
      1, & \text{if}\ \theta=1 \\
      2, & \text{if}\ \theta=0
    \end{cases}
  \end{equation}
\end{theorem}
\begin{proof}
We estimate $aposteriori$ error by dividing the procedure into two parts. In the first part we find error bound corresponding to $velocity$ and $concentration$ followed by the second part estimating error associated with the $pressure$ term. Let us first introduce the residual vector corresponding to each equations 
\[
\textbf{R}=
  \begin{bmatrix}
 \textbf{f}-\{ \rho \frac{\partial \textbf{u}_h}{\partial t} + \rho (\textbf{u}_h \cdot \bigtriangledown) \textbf{u}_h + \bigtriangledown p_h- \nabla \cdot 2 \eta_h \textbf{D}(\textbf{u}_h)  \} \\
    -\bigtriangledown \cdot \textbf{u}_h \\
 g-(\frac{\partial c_h}{\partial t} - \bigtriangledown \cdot \tilde{\bigtriangledown} c_h + \textbf{u} \cdot \bigtriangledown c_h + \alpha c_h )
  \end{bmatrix}
   = 
  \begin{bmatrix}
 \textbf{R}_1\\
  R_2 \\
  R_3
  \end{bmatrix}
\]
where $\textbf{f}=[f_1,f_2]^T$ \vspace{1mm}\\
\textbf{First part:} We have $\forall \hspace{1mm} \textbf{V}$ $\in \bar{\textbf{V}}$
\begin{equation}
\begin{split}
\eta_l \mid \textbf{v} \mid_1^2 + D_{l} \mid d \mid_1^2 + \alpha \|d\|^2 +2 \int_{\Omega} \eta(c,\textbf{u}) \frac{\partial v_1}{\partial y} \frac{\partial v_2}{\partial x} \leq B(\textbf{u},\eta(c,\textbf{u});\textbf{V},\textbf{V}) 
\end{split}
\end{equation}
Since $\textbf{e} \in \bar{\textbf{V}}$ therefore considering $\textbf{e}$ as a test function, adding few required terms to the both sides of (70) and applying $Poincare$ $Friedrics$ inequality on the first term in the left hand side we have modified (70) as follows:
\begin{multline}
\underbrace{\rho (\frac{\partial e_{\textbf{u}}}{\partial t},e_{\textbf{u}})+ (\frac{\partial e_{c}}{\partial t},e_{c})+\eta_l \| e_{\textbf{u}} \|_1^2+  D_{l} \mid e_c\mid_1^2 + \alpha \|e_c\|^2 }_\textit{LHS}\\
 \leq \underbrace{\rho (\frac{\partial e_{\textbf{u}}}{\partial t},e_{\textbf{u}})+ (\frac{\partial e_{c}}{\partial t},e_{c})+ B(\textbf{u},\eta(c,\textbf{u});\textbf{e},\textbf{e}) +
 2 \int_{\Omega} \mid \eta(c,\textbf{u})  \frac{\partial e_{u1}}{\partial y} \frac{\partial e_{u2}}{\partial x} \mid}_\textit{RHS}
\end{multline}
Outline of the proof is finding a lower bound of $\textit{LHS}$ and upper bound for $\textit{RHS}$ of (71) and finally combining them to arrive at the desired estimate. For this purpose applying (29) on the first two terms of the $\textit{LHS}$ after discretizing it with respect to time we have the lower bound as follows:
\begin{multline}
 \frac{\rho}{2 dt}(\|e_{\textbf{u}}^{n+1}\|^2-\|e_{\textbf{u}}^n\|^2)+\frac{1}{2 dt}(\|e_{c}^{n+1}\|^2-\|e_{c}^n\|^2)+\eta_l \| e_{\textbf{u}}^{n,\theta} \|_1^2+  D_{l} \mid e_c^{n,\theta}\mid_1^2 + \alpha \|e_c^{n,\theta}\|^2\\
  \leq \textit{LHS} \leq \textit{RHS}
\end{multline}
Now our aim is to find upper bound for $\textit{RHS}$ through dividing it into three broad parts by splitting the errors and rearranging the terms suitably for further derivation in the following way:
\begin{equation}
\begin{split}
RHS &= [\rho  (\frac{e_{\textbf{u}}^{n+1}-e_{\textbf{u}}^n}{dt}, E^{I,n,\theta}_{\textbf{u}})+(\frac{e_c^{n+1}-e_c^n}{dt},E^{I,n,\theta}_c)+ B(\textbf{u}^n,\eta^n; \textbf{U}^{n,\theta}, E^{I,n,\theta}_{\textbf{U}}) -\\
& \quad  B(\textbf{u}^n_h,\eta^n_h; \textbf{U}^{n,\theta}_h, E^{I,n,\theta}_{\textbf{U}})]+ [\rho  (\frac{e_{\textbf{u}}^{n+1}-e_{\textbf{u}}^n}{dt}, E^{A,n,\theta}_{\textbf{u}})+(\frac{e_c^{n+1}-e_c^n}{dt},E^{A,n,\theta}_c)+\\
& \quad B(\textbf{u}^n,\eta^n; \textbf{U}^{n,\theta}, E^{A,n,\theta}_{\textbf{U}}) - B(\textbf{u}^n_h,\eta^n_h; \textbf{U}^{n,\theta}_h, E^{A,n,\theta}_{\textbf{U}})] + [ 2 \int_{\Omega} \mid \eta^n   \frac{\partial e_{u1}^{n,\theta}}{\partial y} \frac{\partial e_{u2}^{n,\theta}}{\partial x} \mid\\
& \quad -c(e_{\textbf{u}}^{n},\textbf{u}_h^{n,\theta}, e_{\textbf{u}}^{n,\theta}) + a_{PL}(\eta^n-\eta_h^n; \textbf{u}^{n,\theta}_h, e_{\textbf{u}}^{n,\theta})-a_{NLT}(e_{\textbf{u}}^n, c_h^{n,\theta},e_c^{n,\theta})]\\
& = RHS^I+ RHS^A+RHS^E
\end{split}
\end{equation}
The superscripts $I,A,E$ stand for $Interpolation$, $Auxiliary$ and $Extra$ respectively since these three parts correspond to interpolation error, auxiliary error and the extra terms occurred as a consequence of splitting the terms in $\textit{RHS}$ in this way. Now our aim is to bring residual into context and for this purpose we have the following relation on integrating the terms in $B(\cdot,\cdot; \cdot, \cdot)$: $\forall$ $(\textbf{v},q,d) \in \bar{\textbf{V}}$
\begin{multline}
\rho (\frac{e_{\textbf{u}}^{n+1}- e_{\textbf{u}}^{n}}{dt}, \textbf{v})+  (\frac{ e_c^{n+1}-e_c^{n}}{dt}, d)+ B(\textbf{u}^n,\eta^n; \textbf{U}^{n,\theta}, \textbf{V})- B(\textbf{u}^n_h,\eta^n_h; \textbf{U}^{n,\theta}_h, \textbf{V})\\
= \int_{\Omega} \textbf{R}_1^{n,\theta} \cdot \textbf{v} + \int_{\Omega} R_2^{n,\theta} q+ \int_{\Omega} R_3^{n,\theta} d \hspace{20mm}
\end{multline}
On substituting $\textbf{v} ,q,d$ in the above expressions by $E^{I,n,\theta}_{\textbf{u}}, E^{I,n,\theta}_{p}, E^{I,n,\theta}_{c}$ respectively and applying the Cauchy-Schwarz inequality and assumption \textbf{(iii)} on the exact solutions we have the estimated result for $RHS^I$ as follows:
\begin{equation}
\begin{split}
RHS^I & = \int_{\Omega} (\textbf{R}_1^{n,\theta} E^{I,n,\theta}_{\textbf{u}}+ R_2^{n,\theta} E^{I,n,\theta}_{p}+ R_3^{n,\theta} E^{I,n,\theta}_{c} ) \leq h^2( \bar{C}_1 \|\textbf{R}_1^{n,\theta}\| + \bar{C}_2 \|R_2^{n,\theta}\|+ \bar{C}_3 \|R_3^{n,\theta}\|) \\
\end{split}
\end{equation}
The parameters $\bar{C}_i$, for i=1,2,3 are occurred due to  applying assumption  \textbf{(iii)} and the interpolation estimate (27). To estimate $RHS^A$ let us bring here (32) in the following way: $\forall$ $\textbf{V}_h \in V_h \times V_h \times Q_h \times V_h$
\begin{multline}
\rho (\frac{e_{\textbf{u}}^{n+1}- e_{\textbf{u}}^{n}}{dt}, \textbf{v}_h)+  (\frac{ e_c^{n+1}-e_c^{n}}{dt}, d_h)+ B(\textbf{u}^n,\eta^n; \textbf{U}^{n,\theta}, \textbf{V}_h)- B(\textbf{u}^n_h,\eta^n_h; \textbf{U}^{n,\theta}_h, \textbf{V}_h)\\
= \sum_{k=1}^{n_{el}} \{(\tau_k' \textbf{R}^{n,\theta}, \mathcal{L}^*(\textbf{u}_h,\eta_h;\textbf{V}_h))_{\Omega_k}+(\tau_k' \textbf{d},\mathcal{L}^*(\textbf{u}_h,\eta_h;\textbf{V}_h))_{\Omega_k} +  \\
((I-\tau_k^{-1}\tau_k) \textbf{R}^{n,\theta}, \textbf{V}_h)_{\Omega_k} +(\tau_k^{-1}\tau_k \textbf{d}, \textbf{V}_h)_{\Omega_k} \} +(\textbf{TE}^{n,\theta}, \textbf{V}_h)
\end{multline}
On substituting $\textbf{V}_h$ by $E^{A,n,\theta}_{\textbf{U}}$ in the above equation we have
\begin{equation}
\begin{split}
RHS^A & =  \sum_{k=1}^{n_{el}} \{(\tau_k' \textbf{R}^{n,\theta}, \mathcal{L}^*(\textbf{u}_h,\eta_h ;E^{A,n,\theta}_{\textbf{U}}))_{\Omega_k}+(\tau_k' \textbf{d},\mathcal{L}^*(\textbf{u}_h,\eta_h ;E^{A,n,\theta}_{\textbf{U}}))_{\Omega_k} +  \\
& \quad  ((I-\tau_k^{-1}\tau_k) \textbf{R}^{n,\theta}, E^{A,n,\theta}_{\textbf{U}})_{\Omega_k} +(\tau_k^{-1}\tau_k \textbf{d}, E^{A,n,\theta}_{\textbf{U}})_{\Omega_k} \} +(\textbf{TE}^{n,\theta},E^{A,n,\theta}_{\textbf{U}})
\end{split}
\end{equation}
The estimation of the terms in $RHS^A$ follows the same way as we have done in the derivation of auxiliary $apriori$ error estimate. Here we use the observation made upon the bounded property of the elements belonging to finite element spaces $V_h$ and $Q_h$ over each element sub-domain $\Omega_k$ for $k=1,2,...,n_{el}$. On expanding each of the terms and applying that observation on the auxiliary error parts and applying assumption \textbf{(iv)} on the viscosity expression, we have the estimations as follows:
\begin{equation}
\begin{split}
 \sum_{k=1}^{n_{el}} (\tau_k' \textbf{R}^{n,\theta}, \mathcal{L}^*(\textbf{u}_h,\eta_h ;E^{A,n,\theta}_{\textbf{U}}))_{\Omega_k} & \leq  \mid \tau_1 \mid \tilde{C}_1^1 \|\textbf{R}_1^{h,n,\theta}\|+ \mid \tau_2 \mid \epsilon_1' \|e_{\textbf{u}}^{n,\theta}\|_1^2 +\\
 & \quad h^2 \frac{ \mid \tau_2 \mid}{ \epsilon_1' } \tilde{C}_1^2 + \mid \tau_3 \mid \tilde{C}_1^3 \|\textbf{R}_3^{h,n,\theta}\| \\
 ((I-\tau_k^{-1}\tau_k) \textbf{R}^{n,\theta}, E^{A,n,\theta}_{\textbf{U}})_{\Omega_k} & \leq  \mid \tau_1 \mid \tilde{C}_2^1 \|\textbf{R}_1^{n,\theta}\| + \mid \tau_3 \mid \tilde{C}_2^3 \|\textbf{R}_3^{n,\theta}\|
\end{split}
\end{equation}
In the next terms the matrix $\textbf{d}$= $\sum_{i=1}^{n+1}(\frac{1}{dt}M\tau_k')^i(\textbf{F} -M\partial_t \textbf{U}_h - \mathcal{L}(\textbf{u}_h;\textbf{U}_h))=\sum_{i=1}^{n+1}(\frac{1}{dt}M\tau_k')^i \textbf{R}$. Therefore it is clear from this expression that the next two terms can be estimated in the above mentioned way and their estimated results are as follows:
\begin{equation}
\begin{split}
(\tau_k' \textbf{d},\mathcal{L}^*(\textbf{u}_h,\eta_h ;E^{A,n,\theta}_{\textbf{U}}))_{\Omega_k} & \leq \mid \tau_1 \mid \tilde{C}_3^1 \|\textbf{R}_1^{n,\theta}\| + \mid \tau_3 \mid \tilde{C}_3^3 \|\textbf{R}_3^{n,\theta}\| \\
  \sum_{k=1}^{n_{el}}(\tau_k^{-1} \tau_k' \textbf{d},E^{A,n,\theta}_{\textbf{U}})_{k} 
 & \leq \mid \tau_1 \mid \tilde{C}_4^1 \|\textbf{R}_1^{h,n,\theta}\| + \mid \tau_3 \mid \tilde{C}_4^3 \|\textbf{R}_3^{h,n,\theta}\|
\end{split}
\end{equation}
where the parameters $\tilde{C}_1^i, \tilde{C}_3^i$ for $i=1,...,4$ contain bounds of auxiliary error part of each variable over each sub-domain. In the next step we have carried out the estimation of the terms in $RHS^E$ in the same way as we have proceeded in the earlier section during the proof of Theorem 2.
\begin{equation}
\begin{split}
RHS^E & \leq  2 \int_{\Omega} \mid \eta^n   \frac{\partial e_{u1}^{n,\theta}}{\partial y} \frac{\partial e_{u2}^{n,\theta}}{\partial x} \mid + \mid c(e_{\textbf{u}}^{n},\textbf{u}_h^{n,\theta}, e_{\textbf{u}}^{n,\theta}) \mid + \mid a_{PL}(\eta^n-\eta_h^n; \textbf{u}^{n,\theta}_h, e_{\textbf{u}}^{n,\theta}) \mid + \\
& \quad \mid a_{NLT}(e_{\textbf{u}}^n, c_h^{n,\theta},e_c^{n,\theta}) \mid \leq (C_1'+C_2' \eta_s)\|e^{n,\theta}_{\textbf{u}}\|_1^2
\end{split}
\end{equation}
where $C_i'$ for i=1,2 are parameters dependent upon the computed solutions.
Now the terms involving truncation error has to be estimated in slightly different way as we have done in the previous section.  Applying the Cauchy-Schwarz and Young's inequality we have the estimation in the following way:
\begin{equation}
\begin{split}
(\textbf{TE}^{n,\theta}, E^{A,n,\theta}_{\textbf{U}}) & = (\textbf{TE}^{n,\theta}, \textbf{e}^{n,\theta}) - (\textbf{TE}^{n,\theta}, E^{I,n,\theta}_\textbf{U}) \\
& \leq \frac{1}{\epsilon_2'} \|\textbf{TE}^{n,\theta}\|^2 + \frac{\epsilon_2'}{2} (\|\textbf{e}^{n,\theta}\|^2 + \|E^{I,n,\theta}_{\textbf{U}} \|^2) \\
& \leq  \frac{1}{\epsilon_2'} \|\textbf{TE}^{n,\theta}\|^2 +  \frac{\epsilon_2'}{2} \{ \|\textbf{e}^{n,\theta}\|^2 + h^4 (\frac{1+\theta}{2} \|\textbf{U}^{n+1}\|_2 + \frac{1-\theta}{2} \|\textbf{U}^n\|_2)^2 \} \\
& \leq \frac{1}{\epsilon_2'} \|\textbf{TE}^{n,\theta}\|^2 +  \frac{\epsilon_2'}{2} \|\textbf{e}^{n,\theta}\|^2_1 +  h^4 \frac{\epsilon_2'}{2} \bar{C}_5
\end{split}
\end{equation}
Let us now combine all the estimated results (73)-(81) in (72) and multiplying both sides by $2 dt$ and  taking summation over the time steps for $n=0,...,(N-1)$, we finally have
\begin{multline}
\rho \sum_{n=0}^{N-1} ( \|e^{n+1}_{\textbf{u}}\|^2-\|e^{n}_{\textbf{u}}\|^2)+\sum_{n=0}^{N-1} ( \|e^{n+1}_c\|^2-\|e^{n}_c\|^2)+ (2 \eta_l-C_1'-C_2'\eta_s-\epsilon_1' \mid \tau_2 \mid- \epsilon_2')\\
 \sum_{n=0}^{N-1} \|e^{n,\theta}_{\textbf{u}}\|_1^2 dt+ (2D_l-\epsilon_2') \sum_{n=0}^{N-1} \mid e^{n,\theta}_{c}\mid_1^2 dt+ (2 \alpha-\epsilon_2') \sum_{n=0}^{N-1} \|e^{n,\theta}_{c}\|^2 dt \\
 \leq h^2 \sum_{n=0}^{N-1}( \bar{C}_1 \|\textbf{R}_1^{n,\theta}\| + \bar{C}_2 \|R_2^{n,\theta}\|+ \bar{C}_3 \|R_3^{n,\theta}\| + \frac{ \mid \tau_2 \mid}{ \epsilon_1' } \tilde{C}_1^2+  h^2 \frac{\epsilon_2'}{2} \bar{C}_5) dt +\\
 \quad  \mid \tau_1 \mid  \sum_{n=0}^{N-1} \sum_{i=1}^4 \tilde{C}_1^i \|\textbf{R}_1^{n,\theta} \| dt + \mid \tau_3 \mid  \sum_{n=0}^{N-1} \sum_{i=1}^4 \tilde{C}_3^i \|R_3^{n,\theta} \| dt +\frac{1}{\epsilon_2'}  \sum_{n=0}^{N-1} \|\textbf{TE}^{n,\theta}\|^2 dt
\end{multline}
Choose the arbitrary parameters in such a way that all the coefficients in the left hand side can be made positive.
Then taking minimum over the coefficients in the left hand side let us divide both sides by them. Using property (3.7) associated with both implicit time discretisation scheme and the fact that $\tau_1, \tau_3$ are of order $h^2$, we have arrived at the following relation:
\begin{equation}
\boxed{\|\textbf{u}-\textbf{u}_h\|_{\textbf{M}}^2  + \|c-c_h\|_{\textbf{M}}^2 \leq C'(\textbf{R}) (h^2+dt^{2r})}
\end{equation}
where
\begin{equation}
    r=
    \begin{cases}
      1, & \text{if}\ \theta=1 \hspace{1mm} for  \hspace{1mm} backward  \hspace{1mm} Euler  \hspace{1mm} rule \\
      2, & \text{if}\ \theta=0  \hspace{1mm} for  \hspace{1mm} Crank-Nicolson  \hspace{1mm} scheme
    \end{cases}
  \end{equation}
  This only completes one part of $aposteriori$ estimation and in the next part we combine the corresponding pressure part.\vspace{2mm}\\
\textbf{Second part:} Using (59) we can rewrite (58)
\begin{equation}
\begin{split}
b(\textbf{v}_h, I^h_p p-p_h) & = (\frac{\partial e_{\textbf{u}}}{\partial t}, \textbf{v}_h)+ c(e_{\textbf{u}},\textbf{u},\textbf{v}_h)+ c(\textbf{u}_h, e_{\textbf{u}}, \textbf{v}_h) + a_{PL}(\eta; \textbf{u},\textbf{v}_h)-a_{PL}(\eta_h; \textbf{u}_h,\textbf{v}_h)
\end{split}
\end{equation}
On discretizing with respect to time and rearranging the terms in the following way to bring the error part in derivation we have
\begin{equation}
\begin{split}
\sum_{n=0}^{N-1} b(\textbf{v}_h, E^{A,n,\theta}_p) dt & = \sum_{n=0}^{N-1} \{ \frac{e_{\textbf{u}}^{n+1}-e_{\textbf{u}}^n}{dt}+ c(e_{\textbf{u}}^n,\textbf{u}^{n,\theta}, \textbf{v}_h)+  c(\textbf{u}^n,\textbf{u}^{n,\theta}, \textbf{v}_h)-   c(e_{\textbf{u}}^n,e_{\textbf{u}}^{n,\theta}, \textbf{v}_h)+\\
& \quad a_{PL}( \eta^n-\eta_h^n; \textbf{u}^{n,\theta}, \textbf{v}_h) + a_{PL}(\eta_h^n; e_{\textbf{u}}^{n\theta}, \textbf{v}_h)\}
\end{split}
\end{equation}
Now applying $Cauchy-Schwarz$ inequality, $Young$s inequality, property \textbf{(b)} of the $trilinear$ form $c(\cdot, \cdot, \cdot)$ and the above result (83) during the process of derivation we finally have the following result:
\begin{equation}
\sum_{n=0}^{N-1}  b(\textbf{v}_h, E_p^{A,n,\theta}) dt \leq \bar{C}'(\textbf{R})(h^2+ dt^{2r}) \|\textbf{v}_h\|_1
\end{equation}
Using this estimate equation (60) becomes
\begin{equation}
\|I_hp-p_h\|^2_{L^2(L^2)} \leq \bar{C}''(\textbf{R})(h^2+ dt^{2r})
\end{equation} 
Now combining the results obtained in the first and second part and applying the interpolation estimate (27) on pressure interpolation term $E^I_p$, we finally arrive at the following
\begin{equation}
\boxed{\|\textbf{u}-\textbf{u}_h\|^2_{\textbf{M}} + \|p-p_h\|_{\textbf{N}}^2  + \|c-c_h\|^2_{\textbf{M}} \leq \bar{C}(\textbf{R}) (h^2+ dt^{2r})}
\end{equation} 
Now this finally completes derivation of $aposteriori$ error estimation. 
\end{proof}
\begin{remark}
These estimations clearly imply that the scheme is $first$ order convergent in space with respect to total norm, whereas in time it is $first$ order convergent for backward Euler time discretization scheme and $second$ order convergent for Crank-Nicolson method.
\end{remark}

\section{Numerical experiment}
This section verifies credibility of the time dependent $ASGS$ method for the coupled system (1)-(4) through a comparative study between few of the stabilized finite element methods. We have carried out this comparison among  the $time$ $independent$ and $time$ $dependent$ $ASGS$ methods based on their numerical performances for different range of Reynolds number. The tabular representations of the numerical data make the comparison more understandable and play an important role to draw a remarkable conclusion at the end. We have divided this section further into two sub sections based on the nature of coupling of the system of equations to include all possible cases of the study. \vspace{1mm}\\
Let us take $\Omega$ to be a square bounded domain (0,1) $\times$ (0,1).   Piecewise continuous linear finite element space is considered for approximating all three variables: velocity, pressure and concentration. Here we need to mention the definitions of the errors with respect to which the performance of the method has been verified. \vspace{1mm}\\
Let $e_{\textbf{u}}=\sqrt{\|\textbf{u}-\textbf{u}_h\|^2_{\textbf{M}}} $, $e_p=\sqrt{ \|p-p_h\|^2_{\textbf{N}}}$ and $e_c= \sqrt{\|c-c_h\|^2_{\textbf{M}}}$, where the respective norms are already defined in section 4.2 and $Total$ indicates sum of all these errors with respect to their respective norms. \vspace{1mm}\\
 The exact solutions for all the cases are taken as follows: \vspace{1mm}\\
$\textbf{u}=(e^{-t} x^2(x-1)^2y(y-1)(2y-1), -e^{-t}x(x-1)(2x-1)y^2(y-1)^2 )$, \vspace{1mm} \\
 $p=e^{-t}(3x^2+3y^2-2)$ and $c=e^{-t} x y (x-1)(y-1)$ 
 \subsection{One way coupling}
 In this case we consider constant viscosity coefficient which implies $\eta$ is independent of the concentration $c$ of the solute. This concentration makes the Reynolds number ($Re$) \textcolor{blue}{to} play a key role in determining the nature of the flow. Depending upon the range of $Re$ we have performed numerical experiments particularly for two different values of $Re$ to provide a rough idea about the performances of the finite element methods. In the following cases we have considered constant diffusion coefficients, $D_1=D_2=0.01$ too and the reaction coefficient is $\alpha=0.01$.  The tables mainly highlight numerical performances of $time$ $independent$ and $time$ $dependent$ $ASGS$ method for this case. \vspace{1mm}\\
 \textbf{(a) Small Reynolds number:} Keeping wide range of $Re$ in mind we have chosen $Re$=1000 to indicate its smaller value. Here we have presented the errors and order of convergences corresponding to each of the variables, such as velocity, pressure and concentration with respect to the specified norms for different values of power law indices $m=1.5,1.0$  and 0.5 under $time$ $independent$ $ASGS$ method in table 1, 3 and 5 respectively. Table 2, 4 and 6 present the same for $time$ $dependent$ $ASGS$ method. Slightly better convergence rate is obtained corresponding to velocity error ($e_{\textbf{u}}$) in the case of $m=1.5$ for $time$ $dependent$ $ASGS$ method, whereas both methods perform equally well in approximating concentration and pressure for all three values of $m$.  \vspace{1mm}\\
\textbf{(b) Large Reynolds number} In this case we have considered comparatively much higher value of $Re=50000$ in order to provide a brief idea about convergence results of the methods for wide range of $Re$. Table 7, 9 and 11 show the errors and rate of convergences under $time$ $independent$ $ASGS$ method for $m=1.5,1.0$ and 0.5 respectively whereas table 8,10 and 12 present the same for $time$ $dependent$ $ASGS$ method. Notably $time$ $dependent$ $ASGS$ method admits far better order of convergence results corresponding to velocity error  in compared with the $time$ $independent$ one for all the three power law indices. Though in this case too both the methods perform equally well in approximating concentration and pressure, due to dominance of pressure error in calculating the total error, the poor performance of $time$ $independent$ $ASGS$ method in approximating velocity is not reflected in the sum. \vspace{1mm}\\
\subsection{Strong coupling}
Here the viscosity coefficient is chosen to be concentration dependent \cite{21}. We have also considered variable diffusion coefficients, $D_1=exp(-t)y^2(y-1)^2(2y-1)^2x^4(x-1)^4$ and $D_2= exp(-t)x^2(x-1)^2(2x-1)^2y^4(y-1)^4$ in this case. In the above section the $time$ $dependent$ $ASGS$ method has shown better performance among the other stabilized methods for all the combinations of Reynolds number and Power law index. Hence we have only examined its performance for strongly coupled system. Table 13, 14 and 15 contain convergence results for this method for $m=1.5,1.0$ and 0.5 respectively.

\begin{table}[]
\centering
\begin{tabular}{|p{5mm}|p{5mm}|p{12mm}|p{10mm}|p{12mm}|p{10mm}|p{12mm}|p{10mm}|p{12mm}|p{10mm}|}
    \hline
  $dt$ & $\frac{1}{h}$ & $e_{\textbf{u}}$ & RoC & $e_c$& RoC & $e_p$ & RoC &Total & RoC\\ [1mm]
 \hline
$\frac{1}{10}$& 10 & 8.61$e^{-3}$ &  & 3.43$e^{-3}$ & & 1.58$e^{-1}$& &1.58$e^{-1}$ & \\[1mm]
  % \hline
$ \frac{1}{20}$&   20 & 6.19$e^{-3}$ & 0.476  & 1.71$e^{-3}$ & 1.006 & 8.32$e^{-2}$ & 0.927 & 8.35$e^{-2}$  & 0.925\\  [1mm]
     %\hline
$  \frac{1}{40}$&   40 & 3.75$e^{-3}$ & 0.721  & 8.85$e^{-4}$ & 0.948 &  4.30$e^{-2}$ & 0.951 & 4.32$e^{-2}$  & 0.950 \\   [1mm]
    %\hline         
$ \frac{1}{80}$ &  80 & 2.30$e^{-3}$ & 0.705  & 4.44$e^{-4}$ & 0.994 & 1.11$e^{-2}$& 0.973 & 2.21$e^{-2}$ & 0.970 \\ 
    \hline      
\end{tabular}
\caption{ Errors and order of convergences under $time$ $independent$ $ASGS$ method for $Re$=1000 and $m=1.5$ at $T=1$}
    \end{table}

\begin{table}
\centering
\begin{tabular}{|p{5mm}|p{5mm}|p{12mm}|p{10mm}|p{12mm}|p{10mm}|p{12mm}|p{10mm}|p{12mm}|p{10mm}|}
    \hline
  $dt$ & $\frac{1}{h}$ & $e_{\textbf{u}}$ & RoC & $e_c$& RoC & $e_p$ & RoC &Total & RoC\\ [1mm]
 \hline
$\frac{1}{10}$& 10 & 5.67$e^{-3}$ &  & 2.42$e^{-3}$ & & 1.58$e^{-1}$& &1.58$e^{-1}$ & \\[1mm]
  % \hline
$ \frac{1}{20}$&   20 & 3.28$e^{-3}$ & 0.777  & 1.31$e^{-3}$ & 0.887 & 8.32$e^{-2}$ & 0.927 & 8.32$e^{-2}$  &0.927\\  [1mm]
     %\hline
$  \frac{1}{40}$&   40 & 1.84$e^{-3}$ & 0.823  & 6.95$e^{-4}$ & 0.912 &  4.30$e^{-2}$ & 0.951& 4.30$e^{-2}$  &0.951 \\   [1mm]
    %\hline         
$ \frac{1}{80}$ &  80 & 9.91$e^{-4}$ & 0.867  & 3.47$e^{-4}$ & 1.003 &2.19$e^{-2}$& 0.973 & 2.19$e^{-2}$ &0.973\\ 
    \hline      
\end{tabular}
\caption{  Errors and order of convergences under $time$ $dependent$ $ASGS$ method for $Re$=1000 and $m=1.5$ at $T=1$}
\end{table}

\begin{table}[]
\centering
\begin{tabular}{|p{5mm}|p{5mm}|p{12mm}|p{10mm}|p{12mm}|p{10mm}|p{12mm}|p{10mm}|p{12mm}|p{10mm}|}
    \hline
  $dt$ & $\frac{1}{h}$ & $e_{\textbf{u}}$ & RoC & $e_c$& RoC & $e_p$ & RoC &Total & RoC\\ [1mm]
 \hline
$\frac{1}{10}$& 10 & 7.60$e^{-3}$ &  & 3.43$e^{-3}$ & & 1.58$e^{-1}$& &1.58$e^{-1}$ & \\[1mm]
  % \hline
$ \frac{1}{20}$&   20 & 3.87$e^{-3}$ & 0.974  & 1.71$e^{-3}$ & 1.006 & 8.32$e^{-2}$ & 0.927 & 8.35$e^{-2}$  & 0.927\\  [1mm]
     %\hline
$  \frac{1}{40}$&   40 & 1.85$e^{-3}$ & 1.06  & 8.85$e^{-4}$ & 0.948 &  4.30$e^{-2}$ & 0.951 & 4.32$e^{-2}$  & 0.951 \\   [1mm]
    %\hline         
$ \frac{1}{80}$ &  80 & 7.94$e^{-4}$ & 1.22  & 4.44$e^{-4}$ & 0.994 & 2.19$e^{-2}$& 0.973 & 2.19$e^{-2}$ & 0.973 \\ 
    \hline      
\end{tabular}
\caption{ Errors and order of convergences under $time$ $independent$ $ASGS$ method for $Re$=1000 and $m=1.0$ at $T=1$}
    \end{table}

\begin{table}[]
\centering
\begin{tabular}{|p{5mm}|p{5mm}|p{12mm}|p{10mm}|p{12mm}|p{10mm}|p{12mm}|p{10mm}|p{12mm}|p{10mm}|}
    \hline
  $dt$ & $\frac{1}{h}$ & $e_{\textbf{u}}$ & RoC & $e_c$& RoC & $e_p$ & RoC &Total & RoC\\ [1mm]
 \hline
$\frac{1}{10}$& 10 & 5.43$e^{-3}$ &  & 2.42$e^{-3}$ & & 1.58$e^{-1}$& &1.58$e^{-1}$ & \\[1mm]
  % \hline
$ \frac{1}{20}$&   20 & 2.82$e^{-3}$ & 0.947  & 1.31$e^{-3}$ & 0.887 & 8.32$e^{-2}$ & 0.927 & 8.33$e^{-2}$  &0.927\\  [1mm]
     %\hline
$  \frac{1}{40}$&   40 & 1.46$e^{-3}$ & 0.946  & 6.96$e^{-4}$ & 0.912 &  4.31$e^{-2}$ & 0.951& 4.31$e^{-2}$  &0.951 \\   [1mm]
    %\hline         
$ \frac{1}{80}$ &  80 & 6.82$e^{-3}$ & 1.10  & 3.47$e^{-4}$ & 1.003 &2.19$e^{-2}$& 0.973 & 2.19$e^{-2}$ &0.973\\ 
    \hline      
\end{tabular}
\caption{  Errors and order of convergences under $time$ $dependent$ $ASGS$ method for $Re$=1000 and $m=1.0$ at $T=1$}
    \end{table}

\begin{table}[]
\centering
\begin{tabular}{|p{5mm}|p{5mm}|p{12mm}|p{10mm}|p{12mm}|p{10mm}|p{12mm}|p{10mm}|p{12mm}|p{10mm}|}
    \hline
  $dt$ & $\frac{1}{h}$ & $e_{\textbf{u}}$ & RoC & $e_c$& RoC & $e_p$ & RoC &Total & RoC\\ [1mm]
 \hline
$\frac{1}{10}$& 10 & 8.39$e^{-3}$ &  & 3.43$e^{-3}$ & & 1.58$e^{-1}$& &1.58$e^{-1}$ & \\[1mm]
  % \hline
$ \frac{1}{20}$&   20 & 5.34$e^{-3}$ & 0.651  & 1.71$e^{-3}$ & 1.006 & 8.32$e^{-2}$ & 0.927 & 8.34$e^{-2}$  & 0.926 \\  [1mm]
     %\hline
$  \frac{1}{40}$&   40 & 2.88$e^{-3}$ & 0.894  & 8.85$e^{-4}$ & 0.948 &  4.30$e^{-2}$ & 0.951 & 4.32$e^{-2}$  & 0.951 \\   [1mm]
    %\hline         
$ \frac{1}{80}$ &  80 & 1.45$e^{-3}$ & 0.991  & 4.44$e^{-4}$ & 0.994 & 1.11$e^{-2}$& 0.973 & 2.20$e^{-2}$ & 0.973 \\ 
    \hline      
\end{tabular}
\caption{  Errors and order of convergences under $time$ $independent$ $ASGS$ method for $Re$=1000 and $m=0.5$ at $T=1$}
    \end{table}

\begin{table}[]
\centering
\begin{tabular}{|p{5mm}|p{5mm}|p{12mm}|p{10mm}|p{12mm}|p{10mm}|p{12mm}|p{10mm}|p{12mm}|p{10mm}|}
    \hline
  $dt$ & $\frac{1}{h}$ & $e_{\textbf{u}}$ & RoC & $e_c$& RoC & $e_p$ & RoC &Total & RoC\\ [1mm]
 \hline
$\frac{1}{10}$& 10 & 5.62$e^{-3}$ &  & 2.42$e^{-3}$ & & 1.58$e^{-1}$& &1.58$e^{-1}$ & \\[1mm]
  % \hline
$ \frac{1}{20}$&   20 & 3.17$e^{-3}$ & 0.825  & 1.31$e^{-3}$ & 0.887 & 8.32$e^{-2}$ & 0.927 & 8.33$e^{-2}$  &0.927\\  [1mm]
     %\hline
$  \frac{1}{40}$&   40 & 1.74$e^{-3}$ & 0.863  & 6.95$e^{-4}$ & 0.912 &  4.30$e^{-2}$ & 0.951& 4.31$e^{-2}$  &0.951 \\   [1mm]
    %\hline         
$ \frac{1}{80}$ &  80 & 8.92$e^{-4}$ & 0.967  & 3.47$e^{-4}$ & 1.003 &2.19$e^{-2}$& 0.973 & 2.19$e^{-2}$ &0.973\\ 
    \hline      
\end{tabular}
\caption{  Errors and order of convergences under $time$ $dependent$ $ASGS$ method for $Re$=1000 and $m=0.5$ at $T=1$}
    \end{table}

\begin{table}[]
\centering
\begin{tabular}{|p{5mm}|p{5mm}|p{12mm}|p{10mm}|p{12mm}|p{10mm}|p{12mm}|p{10mm}|p{12mm}|p{10mm}|}
    \hline
  $dt$ & $\frac{1}{h}$ & $e_{\textbf{u}}$ & RoC & $e_c$& RoC & $e_p$ & RoC &Total & RoC\\ [1mm]
 \hline
$\frac{1}{10}$& 10 & 8.67$e^{-3}$ &  & 3.43$e^{-3}$ & & 1.58$e^{-1}$& &1.58$e^{-1}$ & \\[1mm]
  % \hline
$ \frac{1}{20}$&   20 & 6.49$e^{-3}$ & 0.418  & 1.71$e^{-3}$ & 1.006 & 8.32$e^{-2}$ & 0.927 & 8.35$e^{-2}$  & 0.925\\  [1mm]
     %\hline
$  \frac{1}{40}$&   40 & 4.24$e^{-3}$ & 0.611  & 8.85$e^{-4}$ & 0.948 &  4.30$e^{-2}$ & 0.951 & 4.32$e^{-2}$  & 0.950 \\   [1mm]
    %\hline         
$ \frac{1}{80}$ &  80 & 3.16$e^{-3}$ & 0.428  & 2.26$e^{-4}$ & 0.994 & 1.11$e^{-2}$& 0.973 & 1.13$e^{-2}$ & 0.965 \\ 
    \hline      
\end{tabular}
\caption{ Errors and order of convergences under $time$ $independent$ $ASGS$ method for $Re$=50000 and $m=1.5$ at $T=1$}
    \end{table}

\begin{table}[]
\centering
\begin{tabular}{|p{5mm}|p{5mm}|p{12mm}|p{10mm}|p{12mm}|p{10mm}|p{12mm}|p{10mm}|p{12mm}|p{10mm}|}
    \hline
  $dt$ & $\frac{1}{h}$ & $e_{\textbf{u}}$ & RoC & $e_c$& RoC & $e_p$ & RoC &Total & RoC\\ [1mm]
 \hline
$\frac{1}{10}$& 10 & 5.67$e^{-3}$ &  & 2.42$e^{-3}$ & & 1.58$e^{-1}$& &1.58$e^{-1}$ & \\[1mm]
  % \hline
$ \frac{1}{20}$&   20 & 3.31$e^{-3}$ & 0.777  & 1.31$e^{-3}$ & 0.887 & 8.32$e^{-2}$ & 0.927 & 8.33$e^{-2}$  &0.927\\  [1mm]
     %\hline
$  \frac{1}{40}$&   40 & 1.87$e^{-3}$ & 0.823  & 6.95$e^{-4}$ & 0.912 &  4.30$e^{-2}$ & 0.951& 4.31$e^{-2}$  &0.951 \\   [1mm]
    %\hline         
$ \frac{1}{80}$ &  80 & 1.02$e^{-3}$ & 0.867  & 3.47$e^{-4}$ & 1.003 &2.19$e^{-2}$& 0.973 & 2.19$e^{-2}$ &0.973\\ 
    \hline      
\end{tabular}
\caption{  Errors and order of convergences under $time$ $dependent$ $ASGS$ method for $Re$=50000 and $m=1.5$ at $T=1$}
    \end{table}

\begin{table}[]
\centering
\begin{tabular}{|p{5mm}|p{5mm}|p{12mm}|p{10mm}|p{12mm}|p{10mm}|p{12mm}|p{10mm}|p{12mm}|p{10mm}|}
    \hline
  $dt$ & $\frac{1}{h}$ & $e_{\textbf{u}}$ & RoC & $e_c$& RoC & $e_p$ & RoC &Total & RoC\\ [1mm]
 \hline
$\frac{1}{10}$& 10 & 8.67$e^{-3}$ &  & 3.43$e^{-3}$ & & 1.58$e^{-1}$& &1.58$e^{-1}$ & \\[1mm]
  % \hline
$ \frac{1}{20}$&   20 & 6.46$e^{-3}$ & 0.418  & 1.71$e^{-3}$ & 1.006 & 8.32$e^{-2}$ & 0.927 & 8.35$e^{-2}$  & 0.925\\  [1mm]
     %\hline
$  \frac{1}{40}$&   40 & 4.20$e^{-3}$ & 0.611  & 8.85$e^{-4}$ & 0.948 &  4.30$e^{-2}$ & 0.951 & 4.32$e^{-2}$  & 0.950 \\   [1mm]
    %\hline         
$ \frac{1}{80}$ &  80 & 3.04$e^{-3}$ & 0.428  & 4.44$e^{-4}$ & 0.994 & 2.19$e^{-2}$& 0.973 & 2.21$e^{-2}$ & 0.965 \\ 
    \hline      
\end{tabular}
\caption{  Errors and order of convergences under $time$ $independent$ $ASGS$ method for $Re$=50000 and $m=1.0$ at $T=1$}
    \end{table}

\begin{table}[]
\centering
\begin{tabular}{|p{5mm}|p{5mm}|p{12mm}|p{10mm}|p{12mm}|p{10mm}|p{12mm}|p{10mm}|p{12mm}|p{10mm}|}
    \hline
  $dt$ & $\frac{1}{h}$ & $e_{\textbf{u}}$ & RoC & $e_c$& RoC & $e_p$ & RoC &Total & RoC\\ [1mm]
 \hline
$\frac{1}{10}$& 10 & 5.67$e^{-3}$ &  & 2.42$e^{-3}$ & & 1.58$e^{-1}$& &1.58$e^{-1}$ & \\[1mm]
  % \hline
$ \frac{1}{20}$&   20 & 3.31$e^{-3}$ & 0.777  & 1.31$e^{-3}$ & 0.887 & 8.32$e^{-2}$ & 0.927 & 8.33$e^{-2}$  &0.927\\  [1mm]
     %\hline
$  \frac{1}{40}$&   40 & 1.87$e^{-3}$ & 0.824  & 6.95$e^{-4}$ & 0.912 &  4.30$e^{-2}$ & 0.951& 4.31$e^{-2}$  &0.951 \\   [1mm]
    %\hline         
$ \frac{1}{80}$ &  80 & 1.02$e^{-3}$ & 0.869  & 3.47$e^{-4}$ & 1.003 &2.19$e^{-2}$& 0.973 & 2.19$e^{-2}$ &0.973\\ 
    \hline      
\end{tabular}
\caption{  Errors and order of convergences under $time$ $dependent$ $ASGS$ method for $Re$=50000 and $m=1.0$ at $T=1$}
    \end{table}

\begin{table}[]
\centering
\begin{tabular}{|p{5mm}|p{5mm}|p{12mm}|p{10mm}|p{12mm}|p{10mm}|p{12mm}|p{10mm}|p{12mm}|p{10mm}|}
    \hline
  $dt$ & $\frac{1}{h}$ & $e_{\textbf{u}}$ & RoC & $e_c$& RoC & $e_p$ & RoC &Total & RoC\\ [1mm]
 \hline
$\frac{1}{10}$& 10 & 8.64$e^{-3}$ &  & 3.43$e^{-3}$ & & 1.58$e^{-1}$& &1.58$e^{-1}$ & \\[1mm]
  % \hline
$ \frac{1}{20}$&   20 & 6.33$e^{-3}$ & 0.449  & 1.71$e^{-3}$ & 1.006 & 8.32$e^{-2}$ & 0.927 & 8.35$e^{-2}$  & 0.925\\  [1mm]
     %\hline
$  \frac{1}{40}$&   40 & 3.98$e^{-3}$ & 0.668  & 8.85$e^{-4}$ & 0.948 &  4.30$e^{-2}$ & 0.951 & 4.32$e^{-2}$  & 0.949 \\   [1mm]
    %\hline         
$ \frac{1}{80}$ &  80 & 2.64$e^{-3}$ & 0.596  & 4.44$e^{-4}$ & 0.994 & 2.19$e^{-2}$& 0.973 & 2.21$e^{-2}$ & 0.968 \\ 
    \hline      
\end{tabular}
\caption{  Errors and order of convergences under $time$ $independent$ $ASGS$ method for $Re$=50000 and $m=0.5$ at $T=1$}
    \end{table}

\begin{table}[]
\centering
\begin{tabular}{|p{5mm}|p{5mm}|p{12mm}|p{10mm}|p{12mm}|p{10mm}|p{12mm}|p{10mm}|p{12mm}|p{10mm}|}
    \hline
  $dt$ & $\frac{1}{h}$ & $e_{\textbf{u}}$ & RoC & $e_c$& RoC & $e_p$ & RoC &Total & RoC\\ [1mm]
 \hline
$\frac{1}{10}$& 10 & 5.67$e^{-3}$ &  & 2.42$e^{-3}$ & & 1.58$e^{-1}$& &1.58$e^{-1}$ & \\[1mm]
  % \hline
$ \frac{1}{20}$&   20 & 3.29$e^{-3}$ & 0.782  & 1.31$e^{-3}$ & 0.887 & 8.32$e^{-2}$ & 0.927 & 8.33$e^{-2}$  &0.927\\  [1mm]
     %\hline
$  \frac{1}{40}$&   40 & 1.86$e^{-3}$ & 0.828  & 6.95$e^{-4}$ & 0.912 &  4.30$e^{-2}$ & 0.951& 4.31$e^{-2}$  &0.951 \\   [1mm]
    %\hline         
$ \frac{1}{80}$ &  80 & 1.01$e^{-3}$ & 0.880  & 3.47$e^{-4}$ & 1.003 &2.19$e^{-2}$& 0.973 & 2.19$e^{-2}$ &0.973\\ 
    \hline      
\end{tabular}
\caption{ Errors and order of convergences under $time$ $dependent$ $ASGS$ method for $Re$=50000 and $m=0.5$ at $T=1$}
    \end{table}

 \begin{table}[]
\centering
\begin{tabular}{|p{5mm}|p{5mm}|p{12mm}|p{10mm}|p{12mm}|p{10mm}|p{12mm}|p{10mm}|p{12mm}|p{10mm}|}
    \hline
  $dt$ & $\frac{1}{h}$ & $e_{\textbf{u}}$ & RoC & $e_c$& RoC & $e_p$ & RoC &Total & RoC\\ [1mm]
 \hline
$\frac{1}{10}$& 10 & 4.13$e^{-3}$ &  & 2.13$e^{-3}$ & & 1.60$e^{-1}$& &1.60$e^{-1}$ & \\[1mm]
  % \hline
$ \frac{1}{20}$&   20 & 2.15$e^{-3}$ & 1.366  & 1.08$e^{-3}$ & 0.984 & 8.34$e^{-2}$ & 0.964 & 8.34$e^{-2}$  & 0.942\\  [1mm]
     %\hline
$  \frac{1}{40}$&   40 & 1.27$e^{-3}$ & 1.137  & 5.83$e^{-4}$ & 0.887 &  4.31$e^{-2}$ & 0.958 & 4.31$e^{-2}$  &0.953 \\   [1mm]
    %\hline         
$ \frac{1}{80}$ &  80 & 6.11$e^{-4}$ & 1.286  & 3.17$e^{-4}$ & 0.877 & 2.19$e^{-2}$& 0.979 & 2.19$e^{-2}$ &0.974\\ 
    \hline      
\end{tabular}
\caption{ Errors and order of convergences under $time$ $dependent$ $ASGS$ method for variable viscosity and diffusion coefficients for $m=1.5$ at $T=1$}
    \end{table}  
 
 \begin{table}[]
\centering
\begin{tabular}{|p{5mm}|p{5mm}|p{12mm}|p{10mm}|p{12mm}|p{10mm}|p{12mm}|p{10mm}|p{12mm}|p{10mm}|}
    \hline
  $dt$ & $\frac{1}{h}$ & $e_{\textbf{u}}$ & RoC & $e_c$& RoC & $e_p$ & RoC &Total & RoC\\ [1mm]
 \hline
$\frac{1}{10}$& 10 & 3.95$e^{-3}$ &  & 2.13$e^{-3}$ & & 1.61$e^{-1}$& &1.61$e^{-1}$ & \\[1mm]
  % \hline
$ \frac{1}{20}$&   20 & 1.74$e^{-3}$ & 1.184  & 1.08$e^{-3}$ & 0.984 & 8.35$e^{-2}$ & 0.943 & 8.35$e^{-2}$  & 0.943\\  [1mm]
     %\hline
$  \frac{1}{40}$&   40 & 8.68$e^{-4}$ & 1.004  & 5.83$e^{-4}$ & 0.887 &  4.31$e^{-2}$ & 0.953 & 4.31$e^{-2}$  &0.953 \\   [1mm]
    %\hline         
$ \frac{1}{80}$ &  80 & 3.68$e^{-4}$ & 1.237  & 3.17$e^{-4}$ & 0.877 & 2.19$e^{-2}$& 0.974 & 2.19$e^{-2}$ &0.975\\ 
    \hline      
\end{tabular}
\caption{ Errors and order of convergences under $time$ $dependent$ $ASGS$ method for variable viscosity and diffusion coefficients for $m=1.0$ at $T=1$}
    \end{table}  

\begin{table}[]
\centering
\begin{tabular}{|p{5mm}|p{5mm}|p{12mm}|p{10mm}|p{12mm}|p{10mm}|p{12mm}|p{10mm}|p{12mm}|p{10mm}|}
    \hline
  $dt$ & $\frac{1}{h}$ & $e_{\textbf{u}}$ & RoC & $e_c$& RoC & $e_p$ & RoC &Total & RoC\\ [1mm]
 \hline
$\frac{1}{10}$& 10 & 3.77$e^{-3}$ &  & 2.13$e^{-3}$ & & 1.64$e^{-1}$& &1.64$e^{-1}$ & \\[1mm]
  % \hline
$ \frac{1}{20}$&   20 & 1.46$e^{-3}$ & 1.366  & 1.08$e^{-3}$ & 0.984 & 8.42$e^{-2}$ & 0.964 & 8.42$e^{-2}$  & 0.964\\  [1mm]
     %\hline
$  \frac{1}{40}$&   40 & 6.65$e^{-4}$ & 1.137  & 5.83$e^{-4}$ & 0.887 &  4.33$e^{-2}$ & 0.958 & 4.34$e^{-2}$  &0.958 \\   [1mm]
    %\hline         
$ \frac{1}{80}$ &  80 & 2.73$e^{-4}$ & 1.286  & 3.17$e^{-4}$ & 0.877 & 2.20$e^{-2}$& 0.979 & 2.20$e^{-2}$ &0.979\\ 
    \hline      
\end{tabular}
\caption{ Errors and order of convergences under $time$ $dependent$ $ASGS$ method for variable viscosity and diffusion coefficients for $m=0.5$ at $T=1$}
    \end{table}

\section{Conclusion}
The detailed discussions on the stability and convergence properties of time dependent $ASGS$ formulation for the mathematical model illustrating the transportation of non-Newtonian fluid obeying power-law provide sufficiently comprehensible idea about the stabilized finite element method. Considering dynamic subscales makes the method more consistent and its accuracy in approximating the solution is proven to be the finest one among other stabilized methods. Another consideration of concentration dependent viscosity coefficients transforms the coupled system into a more general one which is applicable in different fields of studies. Carrying out separate derivations for different power law indices designating the crucial shear thinning and shear thickening properties of the non-Newtonian fluid during the error estimates provides a significant concept about handling these terms in further studies of adaptivity. Besides the excellent numerical performance of time dependent $ASGS$ method for different ranges of Reynolds numbers, specially for very high $Re$ makes it more acceptable in application point of view.

\end{document}